\documentclass[hidelinks,onefignum,onetabnum]{siamart220329}
\usepackage{lipsum}
\usepackage{amsfonts}
\usepackage{graphicx}
\usepackage{epstopdf}
\usepackage{algorithm,algorithmic}
\usepackage{verbatim}
\usepackage{multirow}
\usepackage{subfigure}
\usepackage{amsmath,booktabs,ctable,threeparttable}
\usepackage{amssymb,boxedminipage}
\usepackage{microtype}
\allowdisplaybreaks[4]
\ifpdf
  \DeclareGraphicsExtensions{.eps,.pdf,.png,.jpg}
\else
  \DeclareGraphicsExtensions{.eps}
\fi

\newsiamremark{remark}{Remark}
\newsiamremark{hypothesis}{Hypothesis}
\crefname{hypothesis}{Hypothesis}{Hypotheses}
\newsiamthm{claim}{Claim}
\newsiamremark{example}{Example}

\usepackage{enumitem}
\setlist[enumerate]{leftmargin=.5in}
\setlist[itemize]{leftmargin=.5in}

\crefname{hypothesis}{Hypothesis}{Hypotheses}
\usepackage{amsopn}
\DeclareMathOperator{\Diag}{diag}
\DeclareMathOperator{\Rank}{rank}

\DeclareMathOperator{\Trace}{tr}

\DeclareMathOperator{\Range}{\mathcal{R}}

\newcommand{\imagUnit}{\mathrm{i}}
\newcommand{\eConst}{\mathrm{e}}
\newcommand{\norm}[2][2]{\Vert #2 \Vert_{#1}}
\newcommand{\abs}[1]{\vert #1 \vert}

\headers{A CHEBSHEV--JACKSON SERIES BASED BLOCK SS--RR ALGORITHM}{Z. Jia and T. Liu}

\title{A Chebyshev--Jackson series based block SS--RR algorithm for computing partial
eigenpairs of real symmetric matrices\thanks{Submitted to the editors DATE.
\funding{The work was supported by the National Natural Science Foundation of China (NSFC) under
grants 12171273 and 12571404.}}}

\author{Zhongxiao Jia\thanks{Corresponding author. Department of Mathematical Sciences, Tsinghua University, 100084 Beijing, China
  (\email{jiazx@tsinghua.edu.cn}).}
\and Tianhang Liu\thanks{Department of Mathematical Sciences, Tsinghua University, 100084 Beijing, China
  (\email{lth21@mails.tsinghua.edu.cn}).}
}

\usepackage{amsopn}

\begin{document}
\maketitle

\begin{abstract}
This paper considers eigenpair computations of large symmetric matrices with the desired
eigenvalues lying in a given interval using the contour integral-based block SS--RR method, a Rayleigh--Ritz
projection onto a certain subspace generated by moment matrices. Instead of using
a numerical quadrature to approximately
compute the moments by solving a number of large shifted complex linear
systems at each iteration, we make use of the Chebyshev--Jackson (CJ) series expansion to
approximate the moments, which only involves matrix-vector products and avoids expensive
solutions of the linear systems. We prove that the CJ series expansions pointwise
converge to the moments as the series degree increases, but at different convergence rates
depending on point positions and moment orders. These extend the available convergence results
on the zeroth moment of CJ series expansions to higher order ones. Based on the results established,
we develop a CJ--SS--RR algorithm. Numerical experiments illustrate that the new algorithm is more
efficient than the contour integral-based block SS--RR algorithm with the
trapezoidal rule.

\end{abstract}



\section{Introduction}
The problem of computing a number of specific eigenvalues of a large matrix is common in vast scientific and engineering applications.
In this paper, we consider the numerical solution of the following symmetric eigenvalue problem $ A x = \lambda x$ with $A \in \mathbb{R}^{n \times n}$: Given a real interval $[a, b]$ contained in the spectrum interval of $A$, determine all the $n_{ev}$ eigenpairs $(\lambda, x)$ with the eigenvalues $\lambda\in [a, b]$ counting multiplicities. Such kind of eigenvalue problems arises from numerous applications, e.g., dynamic analysis of structures \cite{block_Lanczos}, linear response eigenvalue problems in quantum chemistry \cite{LSEP}, and the density functional theory based electronic structure calculations \cite{Saad_Eigen}, to mention only a few.

Over the past two decades, starting with the initial
the Sakurai--Sugiura (SS) method \cite{SS--Hankel}, called the SS--Hankel method, and the FEAST eigensolver \cite{FEAST2009}, a new class of numerical methods that are based on contour integration and rational filtering has emerged for computing the eigenvalues of a large matrix in a given region and the associated eigenvectors or invariant subspaces.  The SS--Hankel method uses certain complex moment matrices constructed by a
contour integral with the contour containing the region of interest that the eigenvalues lie in, but it and its
block variant \cite{SS_theory2010} are numerically unstable because of the ill-conditioning of
Hankel matrices \cite{beckermann2000}.
To improve numerical stability and compute eigenpairs more accurately and reliably,
based on the Rayleigh--Ritz projection,
the SS--RR method \cite{SS--RR} and its block variant \cite{block_SS--RR}
have been proposed; the block SS--Arnoldi method based on the block Arnoldi process
has also been presented in \cite{SS--Arnoldi}, and it is mathematically equivalent to
the SS-RR method when the moments are exactly computed.
The FEAST eigensolver,
initially introduced by Polizzi \cite{FEAST2009} in 2009 for Hermitian eigenvalue problems,
only computes an approximation to
the zeroth moment, which is the exact spectral projector corresponding to the eigenvalues
inside the interval of interest. It works on subspaces of a fix dimension, and
uses subspace iteration \cite{Golub_Matrix,Parlett_Symmetric,Saad_Eigen,Stewart_Eigen} on an approximate spectral projector to generate a sequence of subspaces \cite{FEASTasSI}, onto which the Rayleigh--Ritz projection is realized. The FEAST eigensolver has been generalized to non-Hermitian matrices \cite{FEAST_non_Hermitian,FEAST_oblique}, and
%the iterative FEAST (IFEAST) eigensolver,
%where iterative methods
%are used to solve the shifted linear systems involved at each
%iteration, has been proposed \cite{IFEAST}.
it can be regarded
as a restarted variant of the block SS--RR method, in which only the zeroth moment is required and
the subspace dimension is no smaller than the number $n_{ev}$ of the desired eigenvalues.

In the SS eigensolver \cite{SS_package} and FEAST eigensolver \cite{FEAST_package},
one needs to use
some suitably chosen numerical quadratures to approximate the zeroth to higher order
moment matrices and the zeroth order moment matrix, respectively, which
requires the solutions of a number of large shifted complex linear systems at each
iteration, where the shifts are the quadrature nodes. Besides, several
researchers have proposed different rational approximation or filtering methods \cite{FEAST_rational_filter,FEAST_Zolotarev,FEAST_LS_rational}, which also require to solve
a number of large shifted linear systems.
These linear systems are generally solved by iterative methods or solved by a direct solver, e.g.,
the sparse LU factorizations \cite{Golub_Matrix,Saad_Linear} if they are computationally feasible.
For truly large matrices, if they are structured, e.g., banded, then the sparse LU factorizations
are preferable and can solve the shifted linear systems very efficiently.
But if they are generally sparse or dense, the storage requirement and computational cost
of the LU factorizations is
prohibitive, and we have no way but use an iterative method, e.g., a Krylov subspace iterative method
such as BiCGStab method \cite{Saad_Linear}, to solve the linear systems.
We suppose that the shifted linear systems are solved iteratively throughout this paper.
However, notice that these shifted linear systems are highly indefinite when
the region of the interest is inside the spectrum of $A$.
It is well known that, for highly indefinite linear systems, Krylov subspace solvers
are generally inefficient and can be too slow. Furthermore, it lacks a
general effective preconditioning technique for highly indefinite linear systems.

Noticing that the contour integral corresponds to a specific step or piecewise continuous function, Jia and Zhang have proposed a new FEAST SVDsolver based on Chebyshev--Jackson (CJ) series expansion \cite{CJ-FEAST-cross,CJ-FEAST-augmented};
instead of numerical quadrature or any
rational filtering, they make use of the polynomial CJ-series expansion to
construct an approximation to the step function and obtain an approximation to the underlying
spectral projector without solving any linear system.
The resulting CJ-FEAST solver is directly adaptable to real symmetric eigenvalue problems.
They have established the pointwise convergence, error estimates of the CJ series, and accuracy estimates for the approximate spectral projector.
The results show that the convergence of CJ polynomial series expansion
is uniquely determined by the distances between the interval ends and the closest eigenvalues inside
and outside the interval and is independent of the interval location in the spectrum of $A$.
By contrast, iterative solutions of the shifted linear systems in numerical quadratures and other rational approximations critically depends on both the interval location and the distances between the ends and the
closest eigenvalues inside and outside, which determine the degrees of indefiniteness and ill conditioning of the
shifted linear systems, respectively, and affect the convergence of Krylov solvers.

In this paper, following CJ series expansion approximations \cite{CJ-FEAST-cross,CJ-FEAST-augmented}
to the step function and the zeroth order moment, we nontrivially extend them to those higher order moments
in the block SS--RR methods, and propose a new block SS--RR method, called the CJ--SS--RR method.
We prove that the CJ series expansion pointwise converges to the moments as the series degree increases, but convergence rates depend on point positions and moment orders.
%These results generalize the available convergence results in \cite{CJ-FEAST-cross} on the zeroth order %moment of CJ series expansion to higher order ones.
Given the property that the subspace generated in the block SS--RR method is a block Krylov subspace \cite{SS_theory2010,SS_review} when the moments are exactly computed,
we extend the results of the block Krylov subspace in \cite{jia_block,Stewart_Eigen} to the CJ--SS--RR method, and establish the convergence results that are different from those of the SS-type methods in \cite{SS_theory2010,SS_theory2016,SS_review}.
These results show that, just as the CJ-FEAST SVDsolver, the approximation accuracy of the CJ series expansion does not depend on the interval location in the spectrum of $A$.
We consider the practical choice of series degree and the issue of more reliable computation of orthonormal
basis of a given search subspace by replacing its possibly ill conditioned basis matrix
formed by the monomial moments with a better conditioned
basis matrix formed by the better conditioned modified Chebyshev moments, such that
the subspace generated by the computed orthonormal basis is a more accurate approximation to
the true search subspace, thereby retaining more effective information and improving
the computational accuracy. Finally,
we will report numerical experiments to demonstrate the effectiveness and efficiency of the CJ--SS--RR method;
in the meantime, we illustrate that the new method
is much more efficient than the contour integral-based block SS--RR method using the numerical quadrature with the trapezoidal rule.

The paper is organized as follows.
In \cref{sec: SS--RR introduction}, we briefly describe the block SS--RR method and its theoretical background.
In \cref{sec: CJ series}, we establish compact quantitative pointwise convergence results on the CJ series expansion.
Then we develop a CJ--SS--RR method in \cref{sec: approximation and convergence} and analyze its convergence.
In \cref{sec: CJ--SS--RR algorithm}, we present a complete CJ--SS--RR method with several practical issues
addressed.
In \cref{sec: numerical experiments}, we report numerical experiments to illustrate the performance of the CJ--SS--RR method and to compare it with the contour integral-based block SS--RR method using the trapezoidal rule.
Finally, we conclude the paper in \cref{sec: conclusion}.

Throughout this paper, denote by $\norm{\cdot}$ the 2-norm of a vector or matrix, by $\mathcal{R}(X)$ the column space of a matrix $X$, by $\lambda(A)$ the set of all the
eigenvalues of $A$, and by $\lambda_{\min}$
and $\lambda_{\max}$ the minimum and maximum eigenvalues of $A$, respectively.
% Denote the $L_\infty$-norm of $f$ on the interval $(a, b)$ as the supremum of $\vert f \vert$ on $(a, b)$, i.e., $\norm[L_\infty(a, b)]{f} = \sup_{t \in (a, b)} \vert f(t) \vert$.
We denote by $(\lambda_i, x_i), i = 1, \ldots, n$ the eigenpairs of $A$ with $\norm{x_i} = 1$, label the eigenvalues in the interval $[a,b]$ as
$b \geq \lambda_1 \geq \cdots \geq \lambda_{n_{ev}} \geq a$, and write $X_{n_{ev}}=(x_1, \ldots, x_{n_{ev}})$, $\Lambda_{n_{ev}}=\Diag\{\lambda_1, \ldots, \lambda_{n_{ev}}\}$ and $\mathcal{X}_{n_{ev}}=\mathcal{R}(X_{n_{ev}})$.

\section{The block SS--RR method}
\label{sec: SS--RR introduction}

Define $\Omega = \{ \vert z-(a+b)/2 \vert \leq (b-a)/2 \} \subset \mathbb{C}$ to be the circle that
contains $[a, b]$ and goes through $a$ and $b$.
Given two positive integers $M, \ell$ and a matrix $V \in \mathbb{R}^{n \times \ell}$ of full column rank,
define the moment matrix
\begin{equation}\label{eq: contour integral of S}
  \widehat{S} = (\widehat{S}_0, \widehat{S}_1, \ldots, \widehat{S}_{M-1})
\end{equation}
with the monomial moments
\begin{equation} \label{moment}
\widehat{S}_k = \frac{1}{2 \pi \imagUnit} \oint_{\partial \Omega} z^k (zI-A)^{-1} V {\rm d} z,\ k=0, 1, \ldots, M-1,
\end{equation}
where $\partial \Omega$ is the positively-oriented boundary of $\Omega$.
The block SS--RR method \cite{block_SS--RR} is the Rayleigh--Ritz projection onto the subspace $\mathcal{R}(\widehat{S})$
and computes the Ritz pairs with respect to the subspace.
Since the exact $\widehat{S}$ is not available, the block SS--RR method uses a numerical quadrature to compute the integrals approximately:
\begin{equation} \label{eq: numerical quadrature}
  \widehat{S}_k \approx S_k:=\sum_{j=1}^q w_j z_j^k (z_jI-A)^{-1} V, \quad k=0, 1, \ldots, M-1,
\end{equation}
where the $w_j$ and $z_j$ are the quadrature weights and nodes, respectively. Therefore,
\begin{equation}\label{apprS}
S=(S_0, S_1, \ldots, S_{M-1})
\end{equation}
is an approximation to $\widehat{S}$ in \eqref{eq: contour integral of S}.

Next we review some basic results and provide some preliminaries.
For brevity, suppose
$[\lambda_{\min}, \lambda_{\max}] =[-1, 1]$, and let $[a,b] \subset [\lambda_{\min}, \lambda_{\max}] $;
otherwise, we use the linear transformation
\begin{equation}\label{eq: linear transformation}
  l(t) = \frac{2t-\lambda_{\max} - \lambda_{\min}}{\lambda_{\max} - \lambda_{\min}}
\end{equation}
to map $[\lambda_{\min}, \lambda_{\max}]$ to $[-1, 1]$.
Define the step function
\begin{displaymath} %\label{eq: step function}
  h(t) = \begin{cases}
    1,           & t \in (a, b) , \\
    \frac{1}{2}, & t \in \{ a, b \}, \\
    0,           & t \in [-1, 1] \setminus [a, b],
  \end{cases}
\end{displaymath}
where $a$ and $b$ are the discontinuity points of $h(t)$, and the values of $h(t)$ at $a$ and $b$
equal the means of respective right and left limits:
\begin{displaymath}
  \frac{1}{2} \left( h(a+0) + h(a-0) \right) = \frac{1}{2} \left( h(b+0) + h(b-0) \right) = h(a) = h(b) = \frac{1}{2}.
\end{displaymath}
The Cauchy integral formula tells us
\begin{displaymath}
  \frac{1}{2 \pi \imagUnit} \oint_{\partial \Omega} z^k (z-t)^{-1} {\rm d} z = \begin{cases}
    t^k, & t \in (a, b), \\
    0,   &t \in [-1, 1] \setminus [a, b].
  \end{cases}
\end{displaymath}
Thus, according to \cite[Theorem 4]{SS_theory2010}, we have
\begin{equation} \label{eq: S_k and R(S)}
  \widehat{S}_k = A^k\widehat{S}_0 = A^kh(A)V, \quad \mathcal{R}(\widehat{S}) = \mathcal{K}_{M}(A, h(A)V),
\end{equation}
where $\mathcal{K}_{M}(A, h(A)V)$ is the block Krylov subspace generated by $A$ and $h(A)V$.
%provided that no eigenvalue of $A$ is equal to $a$ or $b$.

Since $h(A) = X_{n_{ev}}h(\Lambda_{n_{ev}})X_{n_{ev}}^T$, we can write the matrix $\widehat{S}$ defined in \cref{eq: contour integral of S} as
\begin{displaymath} %\label{blockk}
  \widehat{S} = K_M(A, h(A)V) = X_{n_{ev}}K_M(\Lambda_{n_{ev}}, h(\Lambda_{n_{ev}})X_{n_{ev}}^TV),
\end{displaymath}
where $K_M(A, V) = (V, AV, \ldots, A^{M-1}V)$ denotes the block Krylov matrix generated by
$A$ and the starting block $V$.
The following theorem (cf. \cite[Theorem 1]{SS_theory2016}) shows that the block SS--RR method onto $\mathcal{R}(\widehat{S})$ solves the eigenvalue problem under consideration exactly.

\begin{theorem}\label{thm: SS--RR exact}
  Let $V\in\mathbb{R}^{n \times \ell}$. Then
  \begin{displaymath}
    \mathcal{R}(\widehat{S}) \subseteq \mathcal{X}_{n_{ev}},
  \end{displaymath}
  and the equality is attained if and only if $\Rank(\widehat{S})=n_{ev}$.
\end{theorem}

The condition $\Rank(\widehat{S})=n_{ev}$ is equivalent to
\begin{displaymath}
  \Rank\left( K_M(\Lambda_{n_{ev}}, h(\Lambda_{n_{ev}})X_{n_{ev}}^TV) \right) = n_{ev},
\end{displaymath}
which indicates that $M\ell \geq M \cdot \Rank(X_{n_{ev}}^TV) \geq n_{ev}$ and requires that no
multiplicity of $\lambda_i\in [a, b]$ be greater than $\ell$, a necessary condition for a block Krylov to determine the multiplicities of desired eigenvalues \cite{jia_block}.

Particularly, for $V=v \in \mathbb{R}^n$, suppose that $\widehat{S}$ is of column full rank. Then \begin{equation}\label{bkrylov}
  \mathcal{R}(\widehat{S}) = \mathcal{K}_M(A, h(A)v) = \{ p(A)h(A)v \mid p \in \mathcal{P}_{M-1} \},
\end{equation}
where $\mathcal{P}_{M-1}$ is the set of real polynomials of degree not exceeding $M-1$.

\section{The CJ series expansion}
\label{sec: CJ series}

Let $p \in \mathcal{P}_{M-1}$ be a real polynomial, and define the operator $F_d$ as a mapping from $p$ to an approximate expansion of $\tilde{h}(t) = p(t)h(t)$ obtained by the $d$-degree CJ polynomial series expansion:
\begin{equation} \label{eq: F_d}
  \tilde{h}(t) = p(t)h(t) \approx F_d(p)(t) := \frac{c_0}{2} +  \sum_{j=1}^d \rho_{j, d} c_j T_j(t),
\end{equation}
where $T_j$ is the $j$-degree Chebyshev polynomial of the first kind \cite{ApproxIntroduction}, the coefficients
\begin{equation}\label{eq: Chebyshev coefficient}
  c_j = \frac{2}{\pi} \int_{-1}^1 \frac{\tilde{h}(t) T_j(t)}{\sqrt{1-t^2}}  {\rm d} t = \frac{2}{\pi} \int_a^b \frac{p(t) T_j(t)}{\sqrt{1-t^2}} {\rm d} t,
\end{equation}
and the Jackson damping factors (cf. \cite{eig_count_Saad,Jackson_damp})
\begin{equation}\label{eq: Jackson damping factor}
  \rho_{j, d} = \frac{\sin((j+1)\alpha_d)}{(d+2)\sin(\alpha_d)} + \left( 1-\frac{j+1}{d+2} \right) \cos(j\alpha_d), \quad \text{where } \alpha_d = \frac{\pi}{d+2}.
\end{equation}
Since the coefficients $c_j$ are linear with respect to the polynomial $p$, the mapping $F_d$ is linear in $p$ too: for all real numbers $\beta_1, \beta_2$ and real polynomials $p_1, p_2$, it holds that
\begin{equation}\label{eq: linearity of F_d}
 F_d(\beta_1 p_1 + \beta_2 p_2)= \beta_1 F_d(p_1) + \beta_2 F_d(p_2).
\end{equation}

Define the step function $g$ with the period $2\pi$ to be
\begin{equation}\label{eq: g(theta)}
  g(\theta) = \tilde{h}(\cos \theta) = \begin{cases}
    p(\cos \theta), & \theta \in (-\alpha, -\beta) \cup (\beta, \alpha), \\
    \frac{1}{2}p(\cos\theta), & \theta \in \{ \pm \alpha, \pm \beta \}, \\
    0, & \theta \in [-\pi, -\alpha) \cup (-\beta, \beta) \cup (\alpha, \pi],
  \end{cases}
\end{equation}
where
\begin{equation}\label{alphabeta}
\alpha=\arccos(a),\ \beta=\arccos(b)
\end{equation}
and $g(\theta)$ at $\pm\alpha$ and $\pm\beta$
are the means of respective right and left limits, and the trigonometric polynomial
\begin{equation}\label{eq: q_d(theta)}
  q_d(\theta) = F_d(p)(\cos\theta) = \frac{c_0}{2} + \sum_{j=1}^d \rho_{j, d}c_j\cos(j\theta).
\end{equation}

Analogously to Lemma 1.4 of \cite[Section 1.1.2]{ApproxIntroduction}, we have the following result.

\begin{lemma}
  Let $g$ and $q_d$ be defined as \cref{eq: g(theta),eq: q_d(theta)}, respectively. Then
  \begin{equation}\label{eq: convolution}
    q_d(\theta) = \frac{1}{\pi} \int_{-\pi}^\pi g(\theta + \phi) u_d(\phi) {\rm d} \phi,
  \end{equation}
  where
  \begin{equation}\label{eq: kernel function}
    u_d(\phi) = \frac{1}{2} + \sum_{j=1}^d \rho_{j, d} \cos(j\phi).
  \end{equation}
\end{lemma}

\begin{proof}
  Since $g$ is even, from \cref{eq: Chebyshev coefficient} we obtain
  $$
    \frac{1}{\pi} \int_{-\pi}^\pi g(\phi) \cos(j\phi) {\rm d} \phi = c_j, \quad \frac{1}{\pi} \int_{-\pi}^\pi g(\phi) \sin(j\phi) {\rm d} \phi = 0.
  $$
  Substituting the above relations into \cref{eq: q_d(theta)} yields
  \begin{align*}
    q_d(\theta) &= \frac{1}{\pi} \int_{-\pi}^\pi g(\phi) \left( \frac{1}{2} + \sum_{j=1}^d \rho_{j,d} \left( \cos(j\phi)\cos(j\theta) + \sin(j\phi)\sin(j\theta) \right) \right) {\rm d} \phi \\
    &= \frac{1}{\pi} \int_{-\pi}^\pi g(\phi) u_d(\phi - \theta) {\rm d} \phi = \frac{1}{\pi} \int_{-\pi-\theta}^{\pi-\theta} g(\theta + \phi) u_d(\phi) {\rm d} \phi.
  \end{align*}
  Because of the periodicity of $g$ and $u_d$, \cref{eq: convolution} follows directly from the above relation.
\end{proof}

For $u_d$, there are the following properties.

\begin{proposition}
  Let $u_d$ be defined as \cref{eq: kernel function} and $d \geq 2$. Then
  \begin{align}
    & u_d(\phi) \geq 0, \label{eq: u_d nonnegative} \\
    & \frac{1}{\pi} \int_{-\pi}^\pi u_d(\phi) {\rm d} \phi = 1, \label{eq: u_d zero moment} \\
    & \frac{1}{\pi} \int_{-\pi}^\pi \vert \phi \vert u_d(\phi) {\rm d} \phi \leq \frac{\pi^2}{2(d+2)}, \label{eq: u_d first moment} \\
    & \frac{1}{\pi} \int_{-\pi}^\pi \phi^2 u_d(\phi) {\rm d} \phi \leq \frac{\pi^4}{4(d+2)^2}, \label{eq: u_d second moment}\\
    & \frac{1}{\pi} \int_{-\pi}^\pi \phi^4 u_d(\phi) {\rm d} \phi \leq \left( \frac{1}{4} + \frac{(d+1)\pi^2}{16(d+2)^2} \right) \frac{\pi^6}{(d+2)^3} \leq \frac{\pi^6}{2(d+2)^3}. \label{eq: u_d fourth moment}
  \end{align}
\end{proposition}

\begin{proof}
  It is known from \cite[Section 1.1.2]{ApproxIntroduction} that
  \begin{displaymath}
    u_d(\phi) = \left( \sum_{k=0}^{d} t_k \eConst^{\imagUnit k \phi} \right) \left( \sum_{k=0}^{d} t_k \eConst^{-\imagUnit k \phi} \right) \geq 0, \quad \text{where } t_k = \frac{ \sin\frac{k+1}{d+2}\pi }{ \sqrt{ 2\sum_{k=0}^d \sin^2\frac{k+1}{d+2}\pi } },
  \end{displaymath}
  which yields \cref{eq: u_d nonnegative}.
  Relation \cref{eq: u_d zero moment} is from the proof of
   \cite[Theorem 3.1]{CJ-FEAST-cross}.
%  For \cref{eq: u_d zero moment}, since $\frac{1}{\pi} \int_{-\pi}^{\pi}\cos(j\phi) {\rm d} \phi = 0, j = %1, 2, \ldots$, we have
%  \begin{displaymath}
%    \frac{1}{\pi} \int_{-\pi}^\pi u_d(\phi) {\rm d} \phi = \frac{1}{\pi} \int_{-\pi}^\pi \left( %\frac{1}{2} + \sum_{j=1}^d \rho_{j, d} \cos(j\phi) \right) {\rm d} \phi = \frac{1}{\pi} \int_{-\pi}^\pi %\frac{1}{2} {\rm d} \phi = 1.
%  \end{displaymath}
  By \eqref{eq: u_d zero moment} and the Cauchy--Schwarz inequality, relation (1.1.24) of \cite[Section 1.1.2]{ApproxIntroduction} shows that
  $$
  \frac{1}{\pi} \int_{-\pi}^\pi \vert \phi \vert u_d(\phi) {\rm d} \phi \leq\left(\frac{1}{\pi} \int_{-\pi}^\pi \phi^2 u_d(\phi) {\rm d} \phi\right)^{1/2},
  $$
  and (1.1.25) in this book establishes an expression of its bound, whose estimates
  \cref{eq: u_d first moment,eq: u_d second moment} are from the proof of
  \cite[Theorem 3.2]{CJ-FEAST-cross}. Finally,
  \cref{eq: u_d fourth moment} is straightforward
  from the proof of \cite[Theorem 3.2]{CJ-FEAST-cross}.
\end{proof}

The intermediate bound in \cref{eq: u_d fourth moment} can be asymptotically simplified.
For a modestly sized $d$, e.g., say 50, ignoring the higher order term in it gives
\begin{equation}\label{eq: u_d fourth moment asymptotic}
  \frac{1}{\pi} \int_{-\pi}^\pi \phi^4 u_d(\phi) {\rm d} \phi \lesssim \frac{\pi^6}{4(d+2)^3}.
\end{equation}

We next establish quantitative pointwise convergence results of the approximation $q_d$ to $g$.

\begin{theorem}\label{thm: accuracy outside}
  Let $g$ and $q_d$ be as defined in \cref{eq: g(theta),eq: q_d(theta)}, respectively,
  and $\alpha$ and $\beta$ be as defined in \cref{alphabeta}.
  Define
  \begin{displaymath} %\label{eq: delta_theta}
    \Delta_\theta = \min \{ \vert \theta-\alpha \vert , \vert \theta-\beta \vert \},
  \end{displaymath}
  and the $L_\infty$-norm of $g$ on an interval $(\phi_1, \phi_2)$ as
  \begin{displaymath}
    \norm[L_\infty(\phi_1, \phi_2)]{g} = \sup_{\phi\in (\phi_1, \phi_2)} \abs{g(\phi)} .
  \end{displaymath}
  Then for $\theta\in [0,\pi]\setminus [\beta,\alpha]$ and $d \geq 2$, we have
  \begin{equation}\label{eq: accuracy outside}
    |q_d(\theta) - g(\theta) \vert \leq \frac{\pi^6\norm[L_\infty(0, \pi)]{g}}{2\Delta_\theta^4(d+2)^3}.
  \end{equation}
\end{theorem}

\begin{proof}
  % Since \cref{eq: u_d nonnegative} shows the nonnegativity of $u_d$ and $g(\theta) = 0$, we have
  Due to \cref{eq: convolution}, \cref{eq: u_d nonnegative} and $g(\theta) = 0$ for $\theta\in [0,\pi]\setminus [\beta,\alpha]$, we have
  \begin{displaymath}
    | q_d(\theta) - g(\theta) \vert =| q_d(\theta)|= \left\vert \frac{1}{\pi} \int_{-\pi}^\pi g(\theta + \phi) u_d(\phi) {\rm d} \phi \right\vert \leq \frac{1}{\pi} \int_{-\pi}^\pi \vert g(\theta + \phi) \vert u_d(\phi) {\rm d} \phi.
  \end{displaymath}
  According to \cref{eq: g(theta)}, $g(\varphi) = 0$ for $\varphi \in (-\beta, \beta) \cup (\alpha, 2\pi-\alpha)$.
  For $ \vert \phi \vert < \Delta_\theta$, since $\theta \in [0,\pi]\setminus [\beta,\alpha]$, we have $\theta + \phi \in (-\beta, \beta) \cup (\alpha, 2\pi-\alpha)$ and hence $g(\theta + \phi) = 0$.
  Therefore,
  \begin{displaymath}
    \frac{1}{\pi} \int_{-\pi}^\pi \vert g(\theta + \phi) \vert u_d(\phi) {\rm d} \phi = \frac{1}{\pi} \int_{\Delta_\theta \leq \vert \phi \vert \leq \pi} \vert g(\theta + \phi) \vert u_d(\phi) {\rm d} \phi.
  \end{displaymath}
  For $\Delta_\theta \leq \vert \phi \vert \leq \pi$, we have
  \begin{displaymath}
    \vert g(\theta + \phi) \vert \leq \norm[L_\infty(0, \pi)]{g} \leq \norm[L_\infty(0, \pi)]{g} \cdot \frac{\phi^4}{\Delta_\theta^4}.
  \end{displaymath}
  Then combining the above three relations and \cref{eq: u_d fourth moment}, we obtain
  \begin{displaymath}
    | q_d(\theta) - g(\theta) |
    \leq \frac{\norm[L_\infty(0, \pi)]{g}}{\Delta_\theta^4} \frac{1}{\pi} \int_{\Delta_\theta \leq \abs{\phi} \leq \pi} \phi^4 u_d(\phi) {\rm d} \phi
    \leq \frac{\pi^6\norm[L_\infty(0, \pi)]{g}}{2\Delta_\theta^4(d+2)^3},
  \end{displaymath}
  which proves \eqref{eq: accuracy outside}.
\end{proof}

\begin{remark}\label{rk: accuracy determined by behavoir of the difference}
  As is seen from the proof,
  the difference $g(\theta+\phi)-g(\theta)$ on $(-\Delta_{\theta}, \Delta_{\theta})$ is zero for
  $\theta\in [0,\pi]\setminus [\beta,\alpha]$ while, if $p$ in \cref{eq: g(theta)} is not constant, it is nonzero for $\theta\in[\beta, \alpha]$, which
  implies that $|g(\theta+\phi)-g(\theta)|$ may be larger when $\theta\in [\beta, \alpha]$ than it is when $\theta \in [0,\pi]\setminus [\beta,\alpha]$.
  This is indeed the case, as the following two theorems indicate.
\end{remark}

\begin{theorem}\label{thm: accuracy inside}
 With the same notations as in \cref{thm: accuracy outside}, for $\theta \in (\beta, \alpha)$ and $d \geq 2$, we have
 \begin{equation}\label{eq: accuracy inside}
   | q_d(\theta) - g(\theta) \vert \leq \frac{\pi^6\norm[L_\infty(0, \pi)]{g}}{\Delta_{\theta}^4(d+2)^3} + \frac{\pi^4\norm[L_\infty(\beta, \alpha)]{g''}}{8(d+2)^2}.
 \end{equation}
\end{theorem}

\begin{proof}
  From \cref{eq: convolution,eq: u_d zero moment}, we obtain
  \begin{align*}
    q_d(\theta) - g(\theta)
    =& \frac{1}{\pi} \int_{-\pi}^\pi g(\theta + \phi) u_d(\phi) {\rm d} \phi - \frac{1}{\pi} \int_{-\pi}^\pi g(\theta) u_d(\phi) {\rm d} \phi \\
    =& \frac{1}{\pi} \int_{\Delta_{\theta} \leq \vert \phi \vert \leq \pi} \left( g(\theta + \phi) - g(\theta) \right) u_d(\phi) {\rm d} \phi \\
    &+ \frac{1}{\pi} \int_{-\Delta_{\theta}}^{\Delta_{\theta}} \left( g(\theta + \phi) - g(\theta) \right) u_d(\phi) {\rm d} \phi.
  \end{align*}

 Below we analyze the above two terms in $q_d(\theta) - g(\theta)$ separately.
 From the nonnegativity of $u_d$ shown in \cref{eq: u_d nonnegative}, we have
  \begin{displaymath}
    \left\vert \frac{1}{\pi} \int_{\Delta_{\theta} \leq \vert \phi \vert \leq \pi} \left( g(\theta + \phi) - g(\theta) \right) u_d(\phi) {\rm d} \phi \right\vert \leq \frac{1}{\pi} \int_{\Delta_{\theta} \leq \vert \phi \vert \leq \pi} \vert g(\theta + \phi) - g(\theta) \vert u_d(\phi) {\rm d} \phi.
  \end{displaymath}
  For $\Delta_{\theta} \leq \vert \phi \vert \leq \pi$, it holds that
  \begin{displaymath}
    \vert g(\theta+\phi)-g(\theta) \vert \leq 2\norm[L_\infty(0, \pi)]{g} \leq 2\norm[L_\infty(0, \pi)]{g} \frac{\phi^4}{\Delta_{\theta}^4}.
  \end{displaymath}
  Thus from the above relations and \cref{eq: u_d fourth moment}, we obtain
  \begin{equation}\label{eq: accuracy inside for phi large}
    \begin{aligned}
      \left\vert \frac{1}{\pi} \int_{\Delta_{\theta} \leq \vert \phi \vert \leq \pi} \left( g(\theta + \phi) - g(\theta) \right) u_d(\phi) {\rm d} \phi \right\vert
      &\leq \frac{2\norm[L_\infty(0, \pi)]{g}}{\Delta_{\theta}^4} \frac{1}{\pi} \int_{-\pi}^{\pi} \phi^4 u_d(\phi) {\rm d} \phi \\
      &\leq \frac{\pi^6\norm[L_\infty(0, \pi)]{g}}{\Delta_{\theta}^4(d+2)^3}.
    \end{aligned}
  \end{equation}

  On the other hand, since $u_d$ defined in \cref{eq: kernel function} is even, we have
  \begin{displaymath}
    \frac{1}{\pi} \int_{-\Delta_{\theta}}^{\Delta_{\theta}} \left( g(\theta + \phi) - g(\theta) \right) u_d(\phi) {\rm d} \phi = \frac{1}{\pi} \int_0^{\Delta_{\theta}} \left( g(\theta + \phi) + g(\theta - \phi) - 2g(\theta) \right) u_d(\phi) {\rm d} \phi.
  \end{displaymath}
  Let $G(\phi) = g(\theta + \phi) + g(\theta - \phi) - 2g(\theta)$.
  For $ \vert \phi \vert < \Delta_\theta$, we have $\theta \pm \phi \in (\beta, \alpha)$,
  meaning that $g(\theta\pm\phi) = p(\cos(\theta\pm\phi))$ and $G(\phi)$ is thus smooth.
  Since $G(0) = G'(0) = 0$, there is a $\mu \in (0, 1)$ such that
  \begin{displaymath}
    G(\phi) = G(0) + G'(0)\phi + \frac{1}{2}G''(\mu\phi)\phi^2 = \frac{1}{2} \left( g''(\theta + \mu\phi) + g''(\theta - \mu\phi) \right) \phi^2,
  \end{displaymath}
  which shows that
  \begin{displaymath}
    | G(\phi) \vert \leq \norm[L_\infty(\beta, \alpha)]{g''}\phi^2.
  \end{displaymath}
Therefore,
  \begin{displaymath}
    \left\vert \frac{1}{\pi} \int_{-\Delta_{\theta}}^{\Delta_{\theta}} \left( g(\theta + \phi) - g(\theta) \right) u_d(\phi) {\rm d} \phi \right\vert \leq \norm[L_\infty(\beta, \alpha)]{g''} \frac{1}{\pi} \int_0^\pi \phi^2 u_d(\phi) {\rm d} \phi.
  \end{displaymath}
  Making use of \cref{eq: u_d second moment} and observing that the integration region is half, we obtain
  \begin{displaymath}
    \left\vert \frac{1}{\pi} \int_{-\Delta_{\theta}}^{\Delta_{\theta}} \left( g(\theta + \phi) - g(\theta) \right) u_d(\phi) {\rm d} \phi \right\vert \leq \frac{\pi^4 \norm[L_\infty(\beta, \alpha)]{g''}}{8(d+2)^2},
  \end{displaymath}
  which, together with \cref{eq: accuracy inside for phi large}, establishes \cref{eq: accuracy inside}.
\end{proof}

\begin{theorem}\label{thm: accuracy at alpha and beta}
  With the same notations as in \cref{thm: accuracy outside}, let
  \begin{displaymath}
    \Delta_\alpha = \min \{ \alpha - \beta, 2\pi-2\alpha \}, \quad \Delta_\beta = \min \{ \alpha - \beta, 2\beta \}.
  \end{displaymath}
 Then for $d \geq 2$, we have
  \begin{align}
    | q_d(\alpha) - g(\alpha) \vert \leq \frac{3\pi^6\norm[L_\infty(0, \pi)]{g}}{4\Delta_\alpha^4(d+2)^3} + \frac{\pi^2\norm[L_\infty(\beta, \alpha)]{g'}}{4(d+2)}, \label{eq: accuracy at alpha} \\
    | q_d(\beta) - g(\beta) \vert \leq \frac{3\pi^6\norm[L_\infty(0, \pi)]{g}}{4\Delta_\beta^4(d+2)^3} + \frac{\pi^2\norm[L_\infty(\beta, \alpha)]{g'}}{4(d+2)}. \label{eq: accuracy at beta}
  \end{align}
\end{theorem}

\begin{proof}
  Since $g(\alpha) = \frac{1}{2}( g(\alpha+0) + g(\alpha-0) )$, from \cref{eq: convolution,eq: u_d zero moment}, we obtain
  \begin{displaymath}
    q_d(\alpha) - g(\alpha) = \frac{1}{\pi} \int_{0}^\pi (g(\alpha + \phi)-g(\alpha+0)) u_d(\phi) {\rm d} \phi + \frac{1}{\pi} \int_{-\pi}^0 ( g(\alpha + \phi) - g(\alpha-0) ) u_d(\phi) {\rm d} \phi.
  \end{displaymath}
  We first consider the integral over $[0, \pi]$.
  For $0 < \phi < \Delta_\alpha$, we have $\alpha + \phi \in (\alpha, 2\pi-\alpha)$.
  Therefore, it follows from \cref{eq: g(theta)} that $g(\alpha+\phi)=g(\alpha+0)=0$.
  For $\Delta_\alpha < \phi < \pi$, we have
  $$
  \vert g(\alpha+\phi) - g(\alpha+0) \vert = \vert g(\alpha+\phi) \vert \leq \norm[L_\infty(0,\pi)]{g} \frac{\phi^4}{\Delta_\alpha^4}.
  $$
  Then by \cref{eq: u_d fourth moment}, we obtain
  \begin{equation}\label{eq: accuracy at alpha for phi large}
    \begin{aligned}
      \left\vert \frac{1}{\pi} \int_{0}^\pi (g(\alpha + \phi)-g(\alpha+0)) u_d(\phi) {\rm d} \phi \right\vert
      &\leq \frac{\norm[L_\infty(0, \pi)]{g}}{\Delta_\alpha^4} \frac{1}{\pi} \int_{\Delta_\alpha}^\pi \phi^4 u_d(\phi) {\rm d} \phi \\
      &\leq \frac{\norm[L_\infty(0, \pi)]{g}}{2\Delta_\alpha^4} \frac{1}{\pi} \int_{-\pi}^\pi \phi^4 u_d(\phi) {\rm d} \phi \\
      &\leq \frac{\pi^6\norm[L_\infty(0, \pi)]{g}}{4\Delta_\alpha^4(d+2)^3}.
    \end{aligned}
  \end{equation}

  For $-\Delta_\alpha < \phi < 0$, we have  $\alpha + \phi \in (\beta, \alpha)$.
  Since $p$ is smooth, there is a $\mu \in (0, 1)$ such that
  \begin{displaymath}
    | g(\alpha+\phi) - g(\alpha-0) \vert = \vert g'(\alpha+\mu\phi) \vert \vert \phi \vert \leq \norm[L_\infty(\beta, \alpha)]{g'} \vert \phi \vert .
  \end{displaymath}
  Then combining the above relation with \cref{eq: u_d first moment} yields
  \begin{align*}
    \left\vert \frac{1}{\pi} \int_{-\Delta_\alpha}^0 \left( g(\alpha + \phi) - g(\alpha-0) \right) u_d(\phi) {\rm d} \phi \right\vert
    &\leq \norm[L_\infty(\beta, \alpha)]{g'} \frac{1}{\pi} \int_{-\Delta_\alpha}^0 \vert \phi \vert u_d(\phi) {\rm d} \phi \\
    &\leq \frac{\norm[L_\infty(\beta, \alpha)]{g'}}{2} \frac{1}{\pi} \int_{-\pi}^\pi \vert \phi \vert u_d(\phi) {\rm d} \phi \\
    &\leq \frac{\pi^2 \norm[L_\infty(\beta, \alpha)]{g'}}{4(d+2)}.
  \end{align*}
  Similarly to \eqref{eq: accuracy inside for phi large},
we can establish the following bound:
  \begin{displaymath}
    \left\vert \frac{1}{\pi} \int_{-\pi}^{-\Delta_\alpha} \left( g(\alpha + \phi) - g(\alpha-0) \right) u_d(\phi) {\rm d} \phi \right\vert \leq \frac{\pi^6\norm[L_\infty(0, \pi)]{g}}{2\Delta_\alpha^4(d+2)^3}.
  \end{displaymath}
  % Combining the above two relations with \cref{eq: accuracy at alpha for phi large}, we obtain \cref{eq: accuracy at alpha}.
  Thus \cref{eq: accuracy at alpha} follows
  from the above two relations and \cref{eq: accuracy at alpha for phi large}.
  The proof of \cref{eq: accuracy at beta} is analogous.
\end{proof}

\begin{remark}\label{rk: accuracy in the interval}
  If $M=1$, then $p\in\mathcal{P}_{M-1}$ is a constant and the second terms
   in \eqref{eq: accuracy inside}, \eqref{eq: accuracy at alpha} and
   \eqref{eq: accuracy at beta}
   vanish. As a result, the preceding two theorems degenerate to \cite[Theorems 3.2 and 3.3]{CJ-FEAST-cross}, where the error bounds for $\theta \in [0,\pi]$ tend to zero as fast as $(d+2)^{-3}$.
  However, as long as $p$ is non-constant, the $L_\infty$-norms of $g'$ and $g''$ are nonzero, and the second terms in \eqref{eq: accuracy inside}, \eqref{eq: accuracy at alpha} and
   \eqref{eq: accuracy at beta} will dominate the error bounds for sufficiently large $d$, and the
   convergence becomes slow.
\end{remark}

We present an example to illustrate that the convergence rates in above theorems are indeed order optimal.
Take $[a, b] = [-0.2, 0.4] \subset [-1, 1]$ and three points $t=-0.6, -0.2, 0.1$, of which $-0.6$ and $0.1$ are outside and inside $[a, b]$, respectively.
Polynomials $p(t)$ are the Chebyshev polynomials of the first kind on $[a, b]$.
Note that $q_d(\theta)=F_d(p)(\cos\theta)$, $g(\theta) = p(\cos\theta)h(\cos\theta)$, $\alpha=\arccos(-0.2)$, $\beta=\arccos(0.4)$ and $\theta=\arccos(t)$.
We approximately replace the $L_\infty$-norms of $g'$ and $g''$ by $\norm[L_\infty(a, b)]{p'}$ and $\norm[L_\infty(a, b)]{p'} + \norm[L_\infty(a, b)]{p''}$, respectively.
For each of $t$ and $p$, we plot the true errors $ \vert F_d(p)(t)-p(t)h(t) \vert $ and the corresponding bounds for $d=2, 3, \ldots, 10^4$ in \cref{fig: accuracy of q_d approximating g}.

Clearly, our bounds are order optimal since they
accurately reflect the convergence rates of the true errors as the degree
$d$ increases. Precisely, for $t \not\in [a, b]$ or $\deg p=0$, the convergence is as fast as $(d+2)^{-3}$;
for $t\in[a, b]$ and $\deg p \geq 1$, the convergence rates are
the same as the second terms in the right-hand sides of \cref{eq: accuracy inside,eq: accuracy at alpha}.

\begin{figure}[tbhp]
  \label{fig: accuracy of q_d approximating g}
  \newcommand{\txtH}{0.2\textheight}
  \centering
    \subfigure[{$t = -0.6 \not\in [a, b]$}, $\deg p=0$]{
      \includegraphics[height=\txtH]{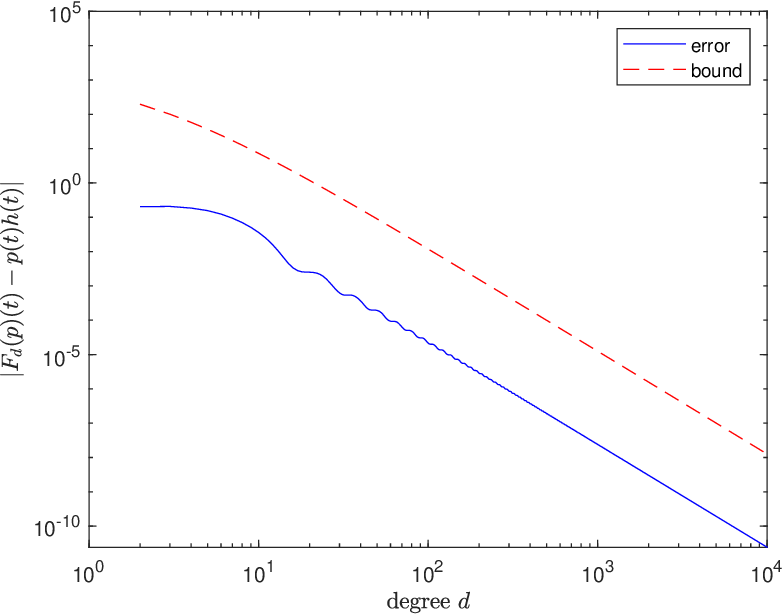}
    }
    \subfigure[{$t = -0.6 \not\in [a, b]$}, $\deg p=1$]{
      \includegraphics[height=\txtH]{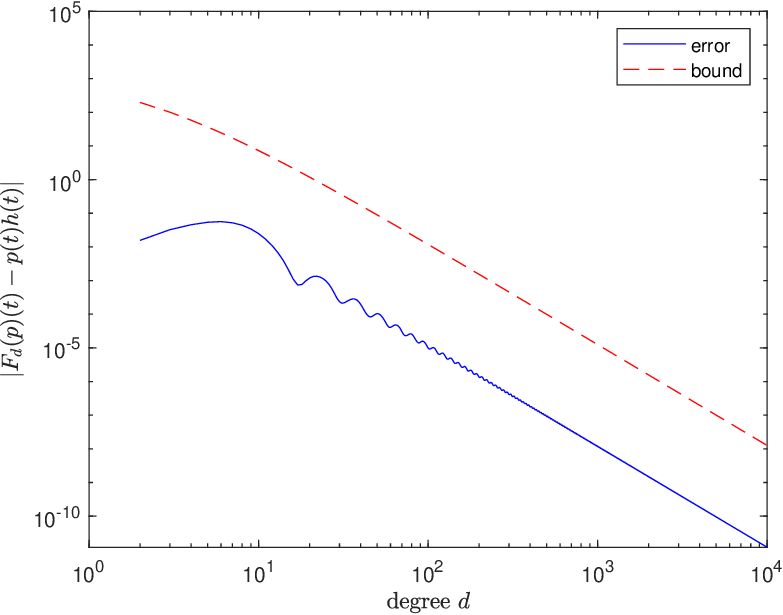}
    }
    \subfigure[$t = 0.1 \in (a, b)$, $\deg p=0$]{
      \includegraphics[height=\txtH]{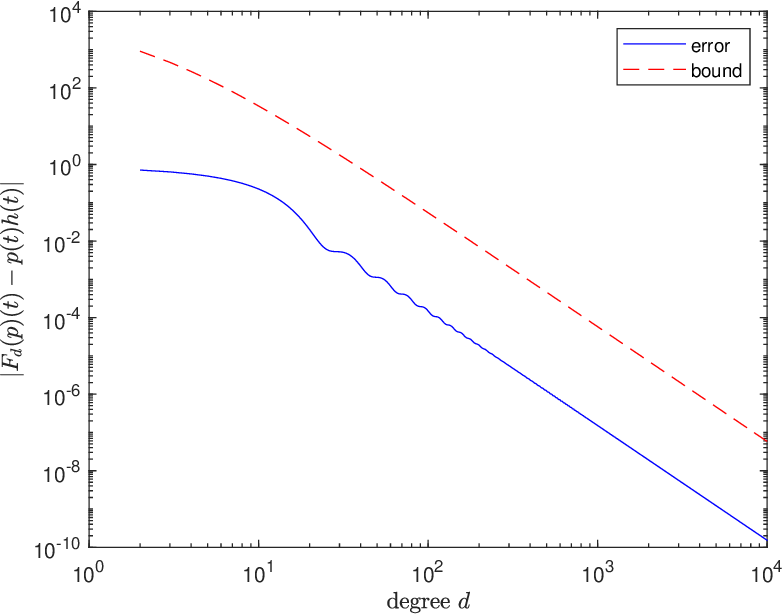}
    }
    \subfigure[$t = 0.1 \in (a, b)$, $\deg p=1$]{
      \includegraphics[height=\txtH]{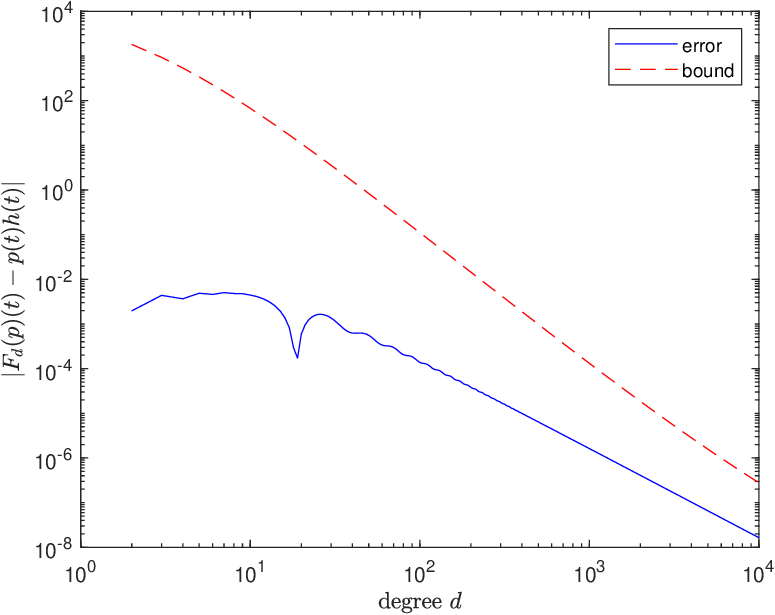}
    }
    \subfigure[$t = -0.2 = a$, $\deg p=0$]{
      \includegraphics[height=\txtH]{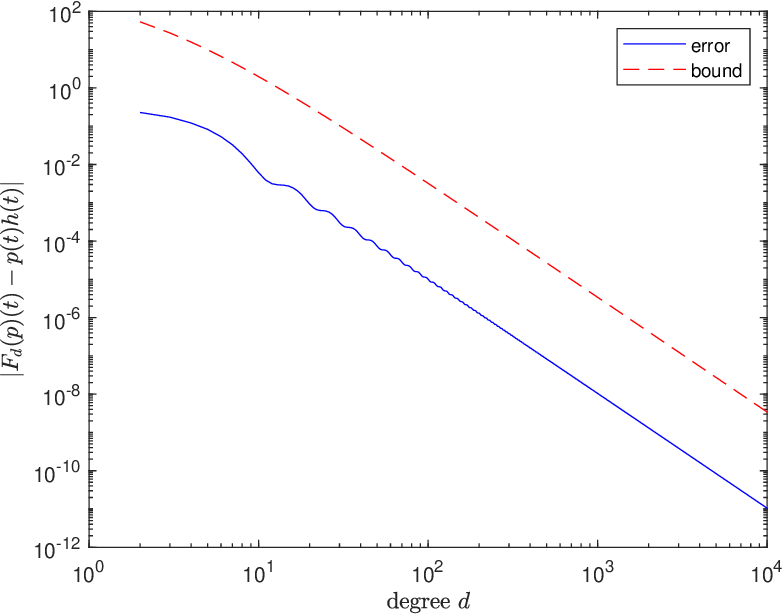}
    }
    \subfigure[$t = -0.2 = a$, $\deg p=1$]{
      \includegraphics[height=\txtH]{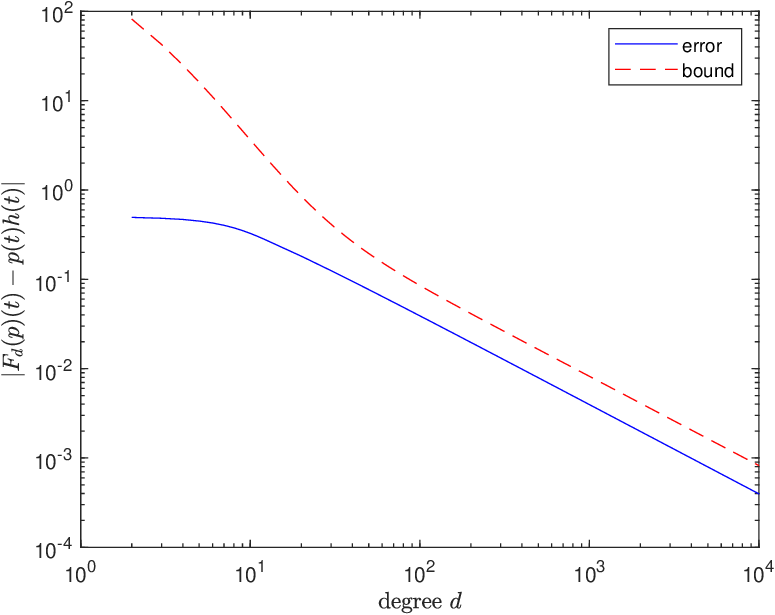}
    }
  \caption{Accuracy of $F_d(p)$ approximating $p\cdot h$.}
\end{figure}

\section{The polynomial approximation in block SS--RR and its convergence}
\label{sec: approximation and convergence}

Although \cref{thm: SS--RR exact} shows that the search subspace $\mathcal{R}(\widehat{S})$ is the desired eigenspace if $\Rank(\widehat{S})=n_{ev}$, the exact computation of integrals $\widehat{S}_k$ in
\eqref{eq: S_k and R(S)} is impractical, so is $p(A)h(A)$. Therefore, we have to consider their approximations and the resulting approximation to $p(A)h(A)$ for $p\in\mathcal{P}_{M-1}$,
and study the convergence of the block SS--RR method.

Given a matrix $V\in\mathbb{R}^{n \times \ell}$, rather than using the rational approximations $S_k$ in
\eqref{eq: numerical quadrature} to construct
an approximation $S$ to $\widehat{S}$ (cf. \eqref{apprS} and \cref{eq: contour integral of S}),
based on the CJ series expansion \eqref{eq: F_d},
we construct a polynomial approximation $S$ to $\widehat{S}$,
where
\begin{equation}\label{eq: polynomial approximation of S}
S=(S_0, S_1, \ldots, S_{M-1})\ \mbox{\ with\ } S_k= F_d(t^k)(A) V \approx A^kh(A)V = \widehat{S}_k.
\end{equation}

Relationships \eqref{eq: S_k and R(S)} and \eqref{eq: polynomial approximation of S}
between $\mathcal{R}(S)$ and the block Krylov subspace $\mathcal{R}(\widehat{S})$ enable us to
establish the following theorem, which is an analogue
of the results on the block Krylov subspace \cite[Lemma 1 and Theorem 2]{jia_block},
where the assumption on the matrix $B$ and the definition of the
index set $N_i$ are exactly the same as those in Lemma 1 of \cite{jia_block} when
the underlying matrix $A$ is symmetric.

\begin{theorem}\label{thm: bound for approx Krylov}
  Let $S$ be defined as \cref{eq: polynomial approximation of S} and $N_i = \{1, 2, \ldots, n\} \setminus \{i, \ldots, i+\ell-1\}$, and suppose that $B = (x_i, \ldots, x_{i+\ell-1})^TV$ is nonsingular, where $x_1,\ldots,x_n$ are the orthonormal eigenvectors of $A$
  corresponding to the eigenvalues $\lambda_1,\ldots,\lambda_n$ labeled in any desired order.
  Then there exists a unique
  $v_i = VB^{-1}e_i \in \Range(V)$ such that
  \begin{equation}\label{vi}
  x_i^Tv_i=1 \ \mbox{\ and\ } x_j^Tv_i=0,j=i+1,\ldots,i+\ell-1
  \end{equation}
   and
  \begin{equation}\label{eq: bound for approx Krylov}
    \tan \angle (\mathcal{R}(S), x_i) \leq \min_{\substack{p \in \mathcal{P}_{M-1} \\ F_d(p)(\lambda_i) \neq 0}} \max_{j \in N_i} \frac{ \vert F_d(p)(\lambda_j) \vert }{ \vert F_d(p)(\lambda_i) \vert } \cdot \tan \angle (v_i, x_i).
  \end{equation}
\end{theorem}

\begin{proof}
  Let $p(t) = \beta_0 + \beta_1 t + \cdots + \beta_{M-1} t^{M-1} \in \mathcal{P}_{M-1}$.
  Then by \cref{eq: linearity of F_d}, we have
  \begin{displaymath}
    F_d(p)(A)v_i = \sum_{j=0}^{M-1} \beta_j F_d(t^j)(A) V B^{-1} e_i = \sum_{j=0}^{M-1} \beta_j S_j B^{-1} e_i \in \mathcal{R}(S).
  \end{displaymath}
  Therefore,
  \begin{equation}\label{eq: tangent bound of R(S) & x_i}
    \tan \angle(\mathcal{R}(S), x_i) = \min_{\substack{u\in\mathcal{R}(S) \\ u \neq 0}} \tan \angle(u, x_i) \leq \min_{\substack{p \in \mathcal{P}_{M-1} \\ F_d(p)(\lambda_i) \neq 0}} \tan \angle(F_d(p)(A)v_i, x_i).
  \end{equation}

Relation \eqref{vi} is the restatement of \cite[Lemma 1 and Theorem 2]{jia_block} for $A$ symmetric.
Therefore, we can express $v_i = x_i + \sum_{j \in N_i}\xi_j x_j$, where the two terms are orthogonal and
the coefficients $\xi_j = x_j^Tv_i$.

Since $F_d(p)(A)x_j = F_d(p)(\lambda_j)x_j$, we have
  \begin{displaymath}
    F_d(p)(A)v_i = F_d(p)(\lambda_i) x_i + \sum_{j \in N_i} \xi_j F_d(p)(\lambda_j) x_j,
  \end{displaymath}
where the two terms are orthogonal. As a result,
\begin{displaymath}
    \begin{aligned}
      \tan^2 \angle (F_d(p)(A)v_i, x_i) & = \sum_{j \in N_i} \frac{ \vert \xi_j F_d(p)(\lambda_j) \vert ^2}{ \vert F_d(p)(\lambda_i) \vert ^2} \\
      &\leq \max_{j \in N_i} \frac{ \vert F_d(p)(\lambda_j) \vert ^2}{ \vert F_d(p)(\lambda_i) \vert ^2} \sum_{j \in N_i} \vert \xi_j \vert ^2 \\
      &= \max_{j \in N_i} \frac{ \vert F_d(p)(\lambda_j) \vert ^2}{ \vert F_d(p)(\lambda_i) \vert ^2} \cdot \tan^2 \angle(v_i, x_i).
    \end{aligned}
  \end{displaymath}
  Rhelation \cref{eq: bound for approx Krylov} follows from the above relation and \cref{eq: tangent bound of R(S) & x_i}.
\end{proof}

To analyze the bound in \cref{eq: bound for approx Krylov}, we next exploit \cref{thm: accuracy outside,thm: accuracy inside,thm: accuracy at alpha and beta} to establish error bounds for the difference between $F_d(p)$ and $p \cdot h$ at the eigenvalues of $A$.
Note that these three theorems
concern the error bounds for $ \vert q_d(\theta)-g(\theta) \vert $, while
our ultimate focus is error bounds for $ \vert F_d(p)(\lambda)-p(\lambda)h(\lambda) \vert $.
To this end, we first present the following two lemmas to estimate the $L_{\infty}$-norms involved.

\begin{lemma}
  Suppose $a < b$ and $M \geq 2$. Then for $p \in \mathcal{P}_{M-1}$, it holds that
  \begin{equation}\label{eq: cor of Markov's theorem}
    \norm[L_\infty(a, b)]{p^{(k)}} \leq \left( \frac{2}{b-a} \right)^k T_{M-1}^{(k)}(1) \cdot \norm[L_\infty(a, b)]{p}, \quad k=1, 2, \ldots, M-1,
  \end{equation}
  where
  \begin{displaymath}
    T_{M-1}^{(k)}(1) = \frac{(M-1)^2((M-1)^2-1)\cdots((M-1)^2-(k-1)^2)}{1 \cdot 3 \cdot 5 \cdots (2k-1)}.
  \end{displaymath}
\end{lemma}

\begin{proof}
  The value of $T_{M-1}^{(k)}(1)$ sees relation (1.97) of \cite{ChebRivlin}.
  For $a=-1$ and $b=1$, \cref{eq: cor of Markov's theorem} is the Markov's theorem (cf. \cite[p. 123]{ChebRivlin}):
  \begin{displaymath}
    \norm[L_\infty(-1, 1)]{p^{(k)}} \leq T_{M-1}^{(k)}(1) \norm[L_\infty(-1, 1)]{p}.
  \end{displaymath}

  Notice that the linear transformation $l(t) = (2t-a-b)/(b-a)$ maps $[a, b]$ to $[-1, 1]$.
  Then for a given $p \in \mathcal{P}_{M-1}$, we have the one-one correspondence
  $p(t) = \tilde{p}(l(t))$ for some $\tilde{p} \in \mathcal{P}_{M-1}$. Therefore,
  \begin{displaymath}
    p^{(k)}(t) = (l'(t))^k \tilde{p}^{(k)}(l(t)) = \left( \frac{2}{b-a} \right)^k \tilde{p}^{(k)}(l(t)),
  \end{displaymath}
which leads to
\begin{displaymath}
  \norm[L_\infty(a, b)]{p^{(k)}} = \left( \frac{2}{b-a} \right)^k \norm[L_\infty(-1, 1)]{\tilde{p}^{(k)}}.
\end{displaymath}
Applying the Markov's theorem to $\tilde{p}$, we obtain
\begin{displaymath}
  \norm[L_\infty(-1, 1)]{\tilde{p}^{(k)}} \leq T_{M-1}^{(k)}(1) \norm[L_\infty(-1, 1)]{\tilde{p}} = T_{M-1}^{(k)}(1) \norm[L_\infty(a, b)]{p}.
\end{displaymath}
Thus \cref{eq: cor of Markov's theorem} follows from the above two relations.
\end{proof}

\begin{lemma}
  Let $-1 \leq a < b \leq 1$, $M \geq 1$ and $p\in\mathcal{P}_{M-1}$.
  Then
  \begin{equation}\label{eq: p' + p'' < p}
    \norm[L_\infty(a, b)]{p'} + \norm[L_\infty(a, b)]{p''} \leq \frac{4(M-1)^4}{(b-a)^2} \norm[L_\infty(a, b)]{p}.
  \end{equation}
\end{lemma}

\begin{proof}
  For $M = 1$, we have $p'(t) \equiv p''(t) \equiv 0$, and thus relation \cref{eq: p' + p'' < p} is
  trivial.
  For $M \geq 2$, we apply \cref{eq: cor of Markov's theorem} to $p$ and get
  \begin{align*}
    \norm[L_\infty(a, b)]{p'} + \norm[L_\infty(a, b)]{p''}
    & \leq \left( \frac{2}{b-a}T_{M-1}'(1) + \frac{4}{(b-a)^2}T_{M-1}''(1) \right) \norm[L_\infty(a, b)]{p} \\
    & = \frac{2(M-1)^2}{(b-a)^2} \left( b-a + \frac{2(M-1)^2-2}{3} \right) \norm[L_\infty(a, b)]{p},
  \end{align*}
  Since $b-a \leq 2$ and $M - 1 \geq 1$, we have
  \begin{displaymath}
    b-a + \frac{2(M-1)^2-2}{3} \leq \frac{2(M-1)^2+4}{3} \leq 2(M-1)^2.
  \end{displaymath}
  From the above two relations, we obtain \cref{eq: p' + p'' < p}.
\end{proof}

With the preceding two lemmas, we have the following theorem.

\begin{theorem}\label{thm: convergence of F(p)}
  Let $\lambda_{ia}, \lambda_{ib}$ and $\lambda_{oa}, \lambda_{ob}$ be the eigenvalues of $A$
  closest to the ends $a$ and $b$ inside and outside $[a, b]$, respectively.
  Define
  \begin{displaymath}
    \Delta_{\min} = \min \{ \vert \alpha-\arccos(\lambda_{ia}) \vert , \vert \alpha-\arccos(\lambda_{oa}) \vert , \vert \beta-\arccos(\lambda_{ib}) \vert , \vert \beta-\arccos(\lambda_{ob}) \vert \},
  \end{displaymath}
  where $\alpha = \arccos(a), \beta = \arccos(b)$.
  Then for $\lambda \in \lambda(A)$ and $d \geq 2$, the error
  \begin{equation}\label{eq: accuracy of eig}
    \left\vert F_d(p)(\lambda) - p(\lambda)h(\lambda) \right\vert \leq \frac{\pi^6 \norm[L_\infty(a, b)]{p}}{\Delta_{\min}^4(d+2)^3} +
    \begin{cases}
      0, & \lambda \not\in [a, b], \\
      \frac{\pi^4 (M-1)^4 \norm[L_\infty(a, b)]{p}}{2(b-a)^2(d+2)^2} , & \lambda \in (a, b), \\
      \frac{\pi^2 (M-1)^2 \norm[L_\infty(a, b)]{p}}{2(b-a)(d+2)}, & \lambda \in \{ a, b \}.
    \end{cases}
  \end{equation}
\end{theorem}

\begin{proof}
  Let $\lambda$ be an eigenvalue of $A$, and denote $\theta=\arccos(\lambda)$.
  We obtain
  \begin{displaymath}
    \Delta_{\min} \leq \min\{ \Delta_\theta, \Delta_\alpha, \Delta_\beta \}.
  \end{displaymath}
It follows from \cref{thm: accuracy outside,thm: accuracy inside,thm: accuracy at alpha and beta} that
  \begin{equation}\label{eq: accurary of eig with theta}
    \left\vert F_d(p)(\lambda) - p(\lambda)h(\lambda) \right\vert \leq
    \frac{\pi^6 \norm[L_\infty(0, \pi)]{g}}{\Delta_{\min}^4(d+2)^3} +
    \begin{cases}
      0, & \lambda \not\in [a, b], \\
      \frac{\pi^4}{8(d+2)^2} \norm[L_\infty(\beta, \alpha)]{g''}, & \lambda \in (a, b), \\
      \frac{\pi^2}{4(d+2)} \norm[L_\infty(\beta, \alpha)]{g'}, & \lambda \in \{ a, b \},
    \end{cases}
  \end{equation}
  where $g(\theta) = p(\cos(\theta))h(\cos(\theta)) = p(\lambda)h(\lambda)$.

  Next, we investigate sizes of the $L_\infty$-norms of the two derivatives in \eqref{eq: accurary of eig with theta}.
  By definition \cref{eq: g(theta)}, we have $\norm[L_\infty(\beta, \alpha)]{g'} = \norm[L_\infty(\beta, \alpha)]{g''} = 0$ for $M = 1$ and
  \begin{equation}\label{eq: g < p}
    \norm[L_\infty(0, \pi)]{g} = \sup_{\phi \in (\beta, \alpha)} \vert p(\cos\phi) \vert = \sup_{t \in (a, b)} \vert p(t) \vert = \norm[L_\infty(a, b)]{p},
  \end{equation}
  and these immediately yield \cref{eq: accuracy of eig} when $M = 1$. Consequently, we only need to consider
  the case $M \geq 2$.
  For $\phi \in (\beta, \alpha)$, by definition \eqref{eq: g(theta)}, we have
  \begin{displaymath}
    g'(\phi) = -p'(\cos\phi)\sin\phi, \quad g''(\phi) = p''(\cos\phi)\sin^2\phi - p'(\cos\phi)\cos\phi,
  \end{displaymath}
  which, together with \cref{eq: cor of Markov's theorem,eq: p' + p'' < p}, yield
  \begin{equation}\label{eq: g' & g'' < p}
    \begin{aligned}
      & \norm[L_\infty(\beta, \alpha)]{g'} \leq \norm[L_\infty(a, b)]{p'} \leq \frac{2(M-1)^2}{b-a} \norm[L_\infty(a, b)]{p}, \\
      & \norm[L_\infty(\beta, \alpha)]{g''} \leq \norm[L_\infty(a, b)]{p'} + \norm[L_\infty(a, b)]{p''} \leq \frac{4(M-1)^4}{(b-a)^2} \norm[L_\infty(a, b)]{p}.
    \end{aligned}
  \end{equation}
  Substituting \cref{eq: g < p,eq: g' & g'' < p} into \cref{eq: accurary of eig with theta} proves \cref{eq: accuracy of eig}.
\end{proof}

Since the ratio $\vert F_d(p)(\lambda_j) / F_d(p)(\lambda_i) \vert$ is independent of the scaling of $p$, i.e.,
$$
  \vert F_d(c \cdot p)(\lambda_j) / F_d(c \cdot p)(\lambda_i) \vert = \vert F_d(p)(\lambda_j) / F_d(p)(\lambda_i) \vert
$$
for any real number $c\not=0$, it is insightful to obtain more from \Cref{thm: convergence of F(p)}: Define
\begin{equation} \label{eq: param: raletive errors respect to p}
  \begin{aligned}
    \gamma_i(p) &= \frac{\pi^6}{\Delta_{\min}^4(d+2)^3} \frac{\norm[L_\infty(a,b)]{p}}{\vert p(\lambda_i) \vert}, \quad
    \delta_i(p)  = \left( \frac{\pi^2(M-1)^2}{(b-a)(d+2)} \right)^2 \frac{\norm[L_\infty(a,b)]{p}}{2 \vert p(\lambda_i) \vert}, \\
    \eta_i(p)   &= \frac{\pi^2(M-1)^2}{(b-a)(d+2)} \frac{\norm[L_\infty(a,b)]{p}}{2 \vert p(\lambda_i) \vert},
  \end{aligned}
\end{equation}
where $\Delta_{\min}$ is defined as in \Cref{thm: convergence of F(p)}. Then
\begin{equation} \label{eq: raletive errors of F_d(p)}
  \begin{aligned}
    \frac{\vert F_d(p)(\lambda_j) \vert}{\vert p(\lambda_i) \vert}
    & \leq h(\lambda_j) \frac{\vert p(\lambda_j) \vert}{\vert p(\lambda_i) \vert} + \gamma_i(p) + \begin{cases}
      0,         & \lambda_j \not\in [a, b], \\
      \delta_i(p), & \lambda_j \in (a, b), \\
      \eta_i(p),   & \lambda_j \in \{a, b\},
    \end{cases} \\
    \frac{\vert F_d(p)(\lambda_i) \vert}{\vert p(\lambda_i) \vert}
    & \geq \begin{cases}
      1 - \gamma_i(p) - \delta_i(p), & \lambda_i \in (a, b), \\
      \frac{1}{2} - \gamma_i(p) - \eta_i(p), & \lambda_i \in \{a, b\}, \\
    \end{cases}
  \end{aligned}
\end{equation}
where $\lambda_i \in [a, b]$.
We see from \cref{eq: param: raletive errors respect to p} that
the parameters $\gamma_i(p)$, $\delta_i(p)$ and $\eta_i(p)$ tend to zero as $d$ increases.
Thus for sufficiently large $d$, the ratio $\vert F_d(p)(\lambda_j) / F_d(p)(\lambda_i) \vert$ is governed by $h(\lambda_j) \vert p(\lambda_j) / p(\lambda_i) \vert$, which vanishes for $\lambda_j \not\in [a, b]$.
% Thus for sufficiently large $d$, the ratio $\vert F_d(p)(\lambda_j) / F_d(p)(\lambda_i) \vert$ is dominated by $h(\lambda_j) \vert p(\lambda_j) / p(\lambda_i) \vert$ which occurs only when $\lambda_j \in [a, b]$.
% Notice that \cref{thm: bound for approx Krylov} requires that $j \in \{1, \ldots, n\} \setminus \{i, \ldots,i+\ell-1\}$, whereas the above analysis suggests that $1 \leq j \leq n_{ev}$ so that $\lambda_j \in [a, b]$.
% Thus for the term $h(\lambda_j) \vert p(\lambda_j) / p(\lambda_i) \vert$, it is adequate to consider the case $j \in \{1, \ldots, n_{ev}\} \setminus \{i, \ldots, i+\ell-1\}$.
In this case, we can present the following result,
which is analogous to \cite[Theorem 3]{jia_block} but is stated in a slightly different form.

\begin{theorem} \label{thm: polynomial's bound of block Krylov}
  Assume that $b \geq \lambda_1 \geq \cdots \geq \lambda_{n_{ev}} \geq a$ and the multiplicities
   of $\lambda_i,\,i=1,\ldots,n_{ev}$ do not exceed $\ell$.
  Given $1 \leq i \leq n_{ev}$, let $\Xi_i$ be the set of the distinct ones among $\lambda_1, \ldots, \lambda_{i-1}$, and suppose that $\lambda_{i-1} \neq \lambda_i$ for $i > 1$ and $ \abs{\Xi_i} \leq M-1$, where $ \abs{\Xi_i} $ denotes the cardinality of $\Xi_i$.
  Define
  \begin{align*}
    & \sigma_i = 1 + \frac{2(\lambda_i-\lambda_{i+\ell})}{\lambda_{i+\ell} - a}, \quad
    m_i(t) = \begin{cases}
      \prod_{\lambda_j \in \Xi_i} (t-\lambda_j), & i \neq 1, \\
      1, & i = 1,
    \end{cases} \\
    & \kappa_i = \frac{\norm[L_\infty(a, b)]{m_i}}{ \vert m_i(\lambda_i) \vert }, \quad
    \varepsilon_i = \begin{cases}
      \kappa_i / T_{M - \abs{\Xi_i} - 1}(\sigma_i), & i=1,\ldots, n_{ev}-\ell, \\
      0, & i=n_{ev}-\ell+1,\ldots,n_{ev}.
      \end{cases}
  \end{align*}
  Then
  \begin{equation}\label{ratiovalue}
    \min_{\substack{p \in \mathcal{P}_{M-1} \\ p(\lambda_i) \neq 0}} \max_{\substack{1 \leq j \leq n_{ev} \\ j \not\in \{i, \ldots, i+\ell-1\}}} \frac{ \vert p(\lambda_j) \vert }{ \vert p(\lambda_i) \vert } \leq \varepsilon_i.
  \end{equation}
\end{theorem}

\begin{proof}
  The case $i = 1, \ldots, n_{ev}-\ell$ has been proved in \cite[Theorem 3]{jia_block}.
  Here we only need to show the case $i = n_{ev}-\ell+1, \ldots, n_{ev}$.
  For $i=n_{ev}-\ell+1,\ldots,n_{ev}$, there is $j \in \{1, \ldots, n_{ev}\} \setminus \{i, \ldots, i+\ell-1\} = \{1, \ldots, i-1\}$ so that $\lambda_j \in \Xi_i$.
  In this case, taking $p(t) = m_i(t)$ gives
  \begin{displaymath}
    \max_{\substack{1 \leq j \leq n_{ev} \\ j \not\in \{i, \ldots, i+\ell-1\}}} \frac{ \vert p(\lambda_j) \vert }{ \vert p(\lambda_i) \vert } = \max_{\lambda_j \in \Xi_i} \frac{ \vert m_i(\lambda_j) \vert }{ \vert m_i(\lambda_i) \vert } = 0 = \varepsilon_i.
  \end{displaymath}
  Therefore, \eqref{ratiovalue} holds for $1\leq i\leq n_{ev}$.
\end{proof}

With the preceding theorems, we can establish the following estimate for the bound in \cref{eq: bound for approx Krylov}.
From \cref{eq: raletive errors of F_d(p)}, we can see that for $\lambda_j \not\in [a, b]$ the upper bound of $\vert F_d(p)(\lambda_j) / p(\lambda_i) \vert$ is $\gamma_i(p)$, independently of the position of $\lambda_j$.
Thus for the concerning interval eigenvalue problem of $A$, we label the eigenvalues of $A$ inside $[a, b]$ in the same order as in \cref{thm: polynomial's bound of block Krylov} and those outside $[a, b]$ at will.

\begin{theorem} \label{thm: upper bound for approx Krylov}
  Label the eigenvalues of $A$ in $[a, b] \subset [\lambda_{\min}, \lambda_{\max}]$ as $\lambda_1, \ldots, \lambda_{n_{ev}}$ with $b \geq \lambda_1 \geq \cdots \geq \lambda_{n_{ev}} \geq a$, suppose that the multiplicity of each of them does not exceed $\ell$, and write eigenvalues outside $[a,b]$ as $\lambda_{n_{ev}+1},\ldots,\lambda_n$ at will.
  Given $1 \leq i \leq n_{ev}$, assume that $\lambda_{i-1} \neq \lambda_i$ for $i > 1$ and the cardinality $ \abs{\Xi_i} \leq M-1$, where $\Xi_i$ is the set of the distinct eigenvalues among $\lambda_1, \ldots, \lambda_{i-1}$.
  Define $\sigma_i$, $m_i(t)$, $\kappa_i$ and $\varepsilon_i$ as in \cref{thm: polynomial's bound of block Krylov}, and
  \begin{align}
    & \nonumber
    \tau_i = \begin{cases}
      \varepsilon_i \cdot T_{M - \abs{\Xi_i} - 1}\left(1 + \frac{2(b-\lambda_{i+\ell})}{\lambda_{i+\ell}-a}\right) , & i=1,\ldots, n_{ev}-\ell, \\
      \kappa_i, & i=n_{ev}-\ell+1,\ldots,n_{ev},
    \end{cases} \\
    & \label{eq: param: raletive errors}
    \hat{\gamma}_i = \frac{\pi^6 \tau_i}{\Delta_{\min}^4(d+2)^3}, \quad
    \hat{\delta}_i = \left( \frac{\pi^2(M-1)^2}{(b-a)(d+2)} \right)^2 \frac{\tau_i}{2}, \quad
    \hat{\eta}_i = \frac{\pi^2(M-1)^2}{(b-a)(d+2)} \frac{\tau_i}{2},
  \end{align}
  where $\Delta_{\min}$ is defined as in \cref{thm: convergence of F(p)}.
  Then for sufficiently large CJ series degree $d$
  such that $\nu_i$ defined in \eqref{consts} is uniformly positive, we have
  \begin{equation} \label{eq: bound for approx Krylov bound}
    \min_{\substack{p \in \mathcal{P}_{M-1} \\ F_d(p)(\lambda_i) \neq 0}} \max_{j \in N_i} \frac{ \vert F_d(p)(\lambda_j) \vert }{ \vert F_d(p)(\lambda_i) \vert }
    \leq
    \frac{\mu_i}{\nu_i},
  \end{equation}
  with $N_i = \{1, 2, \ldots, n\} \setminus \{i, \ldots, i+\ell-1\}$ and
  \begin{equation} \label{consts}
    \mu_i = \hat{\gamma}_i + \max\left\{ \varepsilon_i + \hat{\delta}_i, \frac{\varepsilon_i}{2} + \hat{\eta}_i \right\}, \quad \nu_i = \begin{cases}
      1 - \hat{\gamma}_i - \hat{\delta}_i, &\lambda_i \in (a, b), \\
      \frac{1}{2} - \hat{\gamma}_i - \hat{\eta}_i, &\lambda_i \in \{a, b\}.
    \end{cases}
  \end{equation}
\end{theorem}

\begin{proof}
  \emph{Case I}: $i=1, \ldots, n_{ev}-\ell$.
  For $j \in N_i$, there are three possibilities: $1 \leq j \leq i-1$, $i+\ell \leq j \leq n_{ev}$ and $j > n_{ev}$, which correspond to the cases $\lambda_j \in \Xi_i$, $\lambda_j \in [a, \lambda_{i+\ell}]$ and $\lambda_j \not\in [a, b]$, respectively.
  Let $\hat{l}(t) = 1 + \frac{2(t-\lambda_{i+\ell})}{\lambda_{i+\ell}-a}$ be a linear transformation mapping $[a, \lambda_{i+\ell}]$ to $[-1, 1]$ and
  \begin{displaymath}
    p(t) = m_i(t)T_{M - \abs{\Xi_i} - 1}(\hat{l}(t)).
  \end{displaymath}
  Then $\vert T_{M - \abs{\Xi_i} - 1}(\hat{l}(t)) \vert \leq 1$ for $t \in [a, \lambda_{i+\ell}]$, $\vert T_{M - \abs{\Xi_i} - 1}(\hat{l}(t)) \vert \leq T_{M - \abs{\Xi_i} - 1}(\hat{l}(b))$ for $t \in [a, b]$, and $\hat{l}(\lambda_i) = \sigma_i$.
  Therefore, for $\lambda_j \in [a, \lambda_{i+\ell}]$, we obtain
  \begin{displaymath}
    \frac{\abs{p(\lambda_j)}}{\abs{p(\lambda_i)}} \leq \frac{ \norm[L_\infty(a, b)]{m_i} \cdot \abs{T_{M - \abs{\Xi_i} - 1}(\hat{l}(\lambda_j))} }{ \abs{m_i(\lambda_i)} \cdot T_{M - \abs{\Xi_i} - 1}(\sigma_i) } \leq \varepsilon_i.
  \end{displaymath}
  Since $p(\lambda_j) = 0$ for $\lambda_j \in \Xi_i$ and $h(\lambda_j) = 0$ for $\lambda_j \not\in [a, b]$, we have
  \begin{equation} \label{eq: dominant term of ratio of F_d(p)}
    h(\lambda_j) \frac{\abs{p(\lambda_j)}}{\abs{p(\lambda_i)}} \leq h(\lambda_j) \varepsilon_i \quad \text{for all } j \in N_i.
  \end{equation}
  Notice that the only difference between \cref{eq: param: raletive errors respect to p} and \cref{eq: param: raletive errors} is the substitution of $\frac{\norm[L_\infty(a, b)]{p}}{\abs{p(\lambda_i)}}$ with $\tau_i$. Therefore,
  \begin{displaymath}
    \frac{\norm[L_\infty(a, b)]{p}}{\abs{p(\lambda_i)}} \leq \frac{ \norm[L_\infty(a, b)]{m_i} \cdot T_{M - \abs{\Xi_i} - 1}(\hat{l}(b)) }{ \abs{m_i(\lambda_i)} \cdot T_{M - \abs{\Xi_i} - 1}(\sigma_i) } = \tau_i,
  \end{displaymath}
  indicating that
  \begin{equation} \label{eq: error term of ratio of F_d(p)}
    \gamma_i(p) \leq \hat{\gamma}_i, \quad
    \delta_i(p) \leq \hat{\delta}_i, \quad
    \eta_i(p)   \leq \hat{\eta}_i.
  \end{equation}

  Substituting \cref{eq: dominant term of ratio of F_d(p),eq: error term of ratio of F_d(p)} into \cref{eq: raletive errors of F_d(p)} yields
  \begin{align*}
    \frac{\abs{F_d(p)(\lambda_j)}}{\abs{p(\lambda_i)}}
    & \leq \hat{\gamma}_i + \begin{cases}
      0, & \lambda_j \not\in [a, b], \\
      \varepsilon_i + \hat{\delta}_i, & \lambda_j \in (a, b), \\
      \frac{\varepsilon_i}{2} + \hat{\eta}_i, & \lambda_j \in \{a, b\},
    \end{cases} \quad \text{for } j \in N_i, \\
    \frac{\abs{F_d(p)(\lambda_i)}}{\abs{p(\lambda_i)}}
    & \geq \begin{cases}
      1 - \hat{\gamma}_i - \hat{\delta}_i, & \lambda_i \in (a, b), \\
      \frac{1}{2} - \hat{\gamma}_i - \hat{\eta}_i, & \lambda_i \in \{a, b\},
    \end{cases}
  \end{align*}
  that is,
  \begin{equation} \label{eq: bound of relative errors of F_d(p)}
    \max_{j \in N_i} \frac{\abs{F_d(p)(\lambda_j)}}{\abs{p(\lambda_i)}} \leq \mu_i, \quad
    \frac{\abs{F_d(p)(\lambda_i)}}{\abs{p(\lambda_i)}} \geq \nu_i.
  \end{equation}

  \emph{Case II}: $i=n_{ev}-\ell+1,\ldots,n_{ev}$.
  In this case, for $j \in N_i$, there is $j\in \{1,\ldots,i-1\}$ or $j\in \{i+\ell,\ldots, n\}$
  such that $\lambda_j \in \Xi_i$ or $\lambda_j \not\in [a, b]$.
  Take $p(t) = m_i(t)$.
  Then for $j \in N_i$, we have
  \begin{displaymath}
    h(\lambda_j) \frac{\abs{p(\lambda_j)}}{\abs{p(\lambda_i)}} = 0 = h(\lambda_j) \varepsilon_i
  \end{displaymath}
  since $p(\lambda_j) = 0$ for $\lambda_j \in \Xi_i$ and $h(\lambda_j) = 0$ for $\lambda_j \not\in [a, b]$, and
  \begin{displaymath}
    \frac{\norm[L_\infty(a, b)]{p}}{\abs{p(\lambda_i)}} = \frac{ \norm[L_\infty(a, b)]{m_i} }{ \abs{m_i(\lambda_i)} } = \kappa_i = \tau_i,
  \end{displaymath}
  which yields $\gamma_i(p) = \hat{\gamma}_i$, $\delta_i(p) = \hat{\delta}_i$ and $\eta_i(p) = \hat{\eta}_i$.
  Substitute these relations into \cref{eq: raletive errors of F_d(p)}. We obtain \cref{eq: bound of relative errors of F_d(p)} again.
%  In summary, \cref{eq: bound of relative errors of F_d(p)} holds for $i=1,\ldots,n_{ev}$, which proves \cref{eq: bound for %approx Krylov bound}.
\end{proof}

%\begin{remark}\label{rk: M = 1 the subspace iteration}
%  When $M = 1$ and $M\ell \geq n_{ev}$, there is $i > n_{ev}-\ell$ and $\mu_i = \hat{\gamma}_i$, $\nu_i = %h(\lambda_i) - \hat{\gamma}_i$.
%  In this case, if we reset $N_i = \{\ell+1, \ldots, n\}$, then \cref{eq: bound for approx Krylov bound} %can be written as
%  \begin{displaymath}
%    \max_{j \geq \ell+1} \frac{ \vert F_d(1)(\lambda_j) \vert }{ \vert F_d(1)(\lambda_i) \vert }
%    \leq \frac{\mu_i}{\nu_i} = \frac{\hat{\gamma}_i}{h(\lambda_i) - \hat{\gamma}_i},
%  \end{displaymath}
%  which follows from the convergence of the CJ-FEAST solver based on subspace iteration %\cite{CJ-FEAST-cross}.
%\end{remark}

\begin{remark} \label{rk: param: bound controlers and raletive errors}
  The quantity $\varepsilon_i$ is an upper bound of the ratio $\abs{p(\lambda_j)/p(\lambda_i)}$ and $\tau_i$, which, together with the degree $d$, forms $\hat{\gamma}_i$, $\hat{\delta}_i$ and $\hat{\eta}_i$ and gives a bound for $\norm[L_\infty(a, b)]{p}/\abs{p(\lambda_i)}$.
  Moreover, the parameters $\hat{\gamma}_i$, $\hat{\delta}_i$ and $\hat{\eta}_i$ decide the upper bounds of $ \vert F_d(p)(t) - p(t)h(t) \vert / \vert p(\lambda_i) \vert $ for $t$ outside, inside $[a, b]$ and at the ends $a$ and $b$, respectively.
\end{remark}

\begin{remark}\label{rk: M = 1 the subspace iteration}
   This theorem considers only the case $M \geq 2$.
   For $M = 1$, we set $N_i = \{n_{ev}+1, \ldots, n\}$; \cref{eq: bound for approx Krylov bound} still holds for $\mu_i = \hat{\gamma}_i$ and $\nu_i = h(\lambda_i) - \hat{\gamma}_i$.
   In this manner, relation \cref{eq: bound for approx Krylov bound} follows from the convergence of the FEAST solver based on subspace iteration \cite{CJ-FEAST-cross}.
\end{remark}

%\begin{remark}
%  The parameter $\varepsilon_i$ defined in \cref{eq: param: error control} represents the difficulty of %approximating $p/p(\lambda_i) \cdot h$ by $F_d(p/p(\lambda_i))$.
%  Here $p$ is the specific polynomial selected in the proof.
%  The principle of selecting $p$ is to minimize $\max_{i+\ell \leq j \leq n_{ev}} \vert %p(\lambda_j)/p(\lambda_i) \vert $, which in turn increases $ \vert p(t)/p(\lambda_i) \vert $ when %$\lambda_i < t \leq b$.
%  This makes it harder to approximate $p/p(\lambda_i) \cdot h$ since %$\norm[L_\infty(a,b)]{p/p(\lambda_i)}$ increases.
%  % In this manner, $\varepsilon_i$ reflects the effect of $ \vert p(t)/p(\lambda_i) \vert $ on the %approximation.
%\end{remark}

\begin{remark} \label{rk: convergent analysis for M and d}
  We can write the ratio $\mu_i/\nu_i$ as
  \begin{displaymath}
    \frac{\mu_i}{\nu_i} = \frac{\varepsilon_i}{h(\lambda_i)} + \frac{\varepsilon_i}{h(\lambda_i)} \frac{h(\lambda_i)-\nu_i}{\nu_i} + \frac{\mu_i-\varepsilon_i}{\nu_i}.
  \end{displaymath}
  Since $\lim_{d \to \infty} \mu_i = \varepsilon_i$ and $\lim_{d \to \infty} \nu_i = h(\lambda_i)$, the first term in the above relation reduces to \cref{thm: polynomial's bound of block Krylov} as $d \to \infty$, and it tends to zero as $M$ increases.
  The second and third terms tend to zero as $d$ increases, and are the deviations caused by the CJ series expansion approximation.
  Therefore, the ratio $\mu_i/\nu_i$ decreases as $d$ increases,
  meaning faster convergence.
  However, unlike the series degree $d$, the parameter $M$ is not so free to change
  and must be of small to modest size, and actually one only requires that $M\ell\geq n_{ev}$ slightly in the
  block SS--RR method.
%  However, things will become more complicated once the disturbances are taken into account.
%  If a constant $d$ is given and $M$ is increased with $\nu_i > 0$, then $\varepsilon_i/h(\lambda_i)$ will %decrease while $(h(\lambda_i)-\nu_i)/\nu_i$ and $(\mu_i-\varepsilon_i)/\nu_i$ will increase.
%  At this point, it is difficult to ascertain the change in $\mu_i/\nu_i$.
%  Thus, a prudent approach is to increase $M$ and $d$ together.
%  This implies that the larger $M$ is, the more accurate the approximation needs to be.
\end{remark}

\section{The CJ--SS--RR algorithm}
\label{sec: CJ--SS--RR algorithm}

Recall that the linear transformation \eqref{eq: linear transformation}
maps the spectral interval $[\lambda_{\min}, \lambda_{\max}]$ to $[-1, 1]$.
% It is important to note that, when using this transformation in practice, it is necessary to give the rough estimates for $\lambda_{\max}$ and $\lambda_{\min}$.
When using this transformation in computations, it is necessary to give reasonably good
estimates for $\lambda_{\max}$ and $\lambda_{\min}$.
As commonly suggested in a large number of literature, e.g., \cite{eig_count_Saad,spectral-estimate2011zhou}, one can run the symmetric Lanczos method on $A$ for several steps, say $30\sim 50$,
to get estimates of $\lambda_{\max}$ and $\lambda_{\min}$.
For instance, $\lambda_{\max} \approx \theta_1 + \norm{r_1}$ and $\lambda_{\min} \approx \theta_{s} - \norm{r_{s}}$, where $\theta_1$ and $\theta_s$ are the largest and smallest Ritz values
and $r_1,r_s$ are the corresponding residuals of the $s$-step Lanczos method.
In what follows, we always assume that the spectrum interval
of $A$ has been transformed into $[-1, 1]$.

\subsection{The construction of an orthonormal basis of the search subspace}
Given the search subspace $\mathcal{R}(S)$ with $S$ defined in \cref{eq: polynomial approximation of S},
to obtain approximate eigenpairs, let the columns of $U$ form an orthonormal basis of $\mathcal{R}(S)$. We
use the standard Rayleigh--Ritz projection
\begin{equation}\label{eq: Rayleigh--Ritz}
  U^T A U y = \tilde{\lambda}y, \quad \tilde{x} = Uy \mbox{\ \ with\ \ } \norm{y}=1,
\end{equation}
to obtain the approximate eigenpairs $(\tilde{\lambda}, \tilde{x})$, called
the Ritz pairs of $A$ with respect to $\mathcal{R}(S)$.
Theoretically, a commonly used approach is to compute a rank-revealing
QR decomposition
or the singular value decomposition (SVD) of $S$ to get
an orthonormal basis of $\mathcal{R}(S)$. In finite precision arithmetic, however, notice that
\begin{displaymath}
  S = (F_d(1)(A)V, \ldots, F_d(t^{M-1})(A)V) \approx (h(A)V, Ah(A)V, \ldots, A^{M-1}h(A)V)=\widehat{S}.
\end{displaymath}
The matrix $\widehat{S}$ is typically ill conditioned even for
$M$ small and its condition number increases at exponential rate with $M$ increasing
\cite{beckermann2000}, so does that of $S$. As a consequence, the subspace generated by the {\em computed}
orthonormal basis vectors of $\mathcal{R}(S)$ may be far from the true
$\mathcal{R}(S)$ \cite[pp. 373-4]{higham}.
%, indicating that
%it is generally not good to compute an orthonormal basis $U$ of $\mathcal{R}(S)$
%by computing the rank-revealing QR decomposition or the truncated
%SVD of $S$.
Realize the ill conditioning of $S$. It is common in
the SS--RR type methods \cite{block_SS--RR,SS_theory2016,SS_review} to use the truncated SVD and
take those dominant right singular vectors associated
with the large singular values as a replacement of $\mathcal{R}(S)$, whose dimension
is generally at most a couple of
dozens for $\ell$ small when $\mathcal{R}(S)$ is a good
approximation to $\mathcal{R}(\widehat{S})$ and may be much lower
than the dimension $M\ell$ of $\mathcal{R}(S)$. Consequently, such approach possibly discards
too much effective information, and impairs the strength of the SS--RR method artificially and severely. Even worse, it is well possible that the dimension of effective subspace is smaller than the number $n_{ev}$ of desired eigenpairs when $n_{ev}\gg 1$, thereby
leading to the failure of the block SS--RR method when computing the $n_{ev}$ desired eigenpairs.

In order to fix this prevailing severe implementation deficiency, computationally,
it is very attempting to take new well-conditioned basis vectors of $\mathcal{R}(S)$ rather than the
ill-conditioned ones generated by the monomials $1,t,\ldots,t^{M-1}$ and then orthonormalize
them to obtain a new orthonormal basis vectors that retain all the information of $\mathcal{R}(S)$. In
finite precision arithmetic, such orthonormal basis vectors generate $\mathcal{R}(S)$ to the working
precision when the condition number of the new basis matrix is of size $O(1)$.  The Chebyshev basis vectors of $\mathcal{R}(S)$ are better choices, and they
form a better conditioned matrix. Precisely, we replace $t^k, k = 0, 1, \ldots, M-1$
by the $k$-degree Chebyshev polynomials $p_k = T_k((2t-a-b)/(b-a))$ in \eqref{eq: polynomial approximation of
S} and obtain the modified moments. Numerical experiments in \cref{subsec: condition number of S} will confirm that the new basis vectors are always numerically linearly independent
and thus retain all the information of $\mathcal{R}(S)$.

%Suppose there are $\{p_k\} \subset \mathcal{P}_{M-1}$ and nonsingular matrix $R \in \mathbb{R}^{M \times %M}$, such that
%\begin{displaymath}
%  (p_0(t), p_1(t), \cdots, p_{M-1}(t)) = (1, t^1, \cdots, t^{M-1})R.
%\end{displaymath}
%It follows from the above relation and \cref{eq: linearity of F_d} that
%\begin{displaymath}
%  \widetilde{S} := (F_d(p_0)(A)V, F_d(p_1)(A)V, \cdots, F_d(p_{M-1})(A)V) = S \cdot (R \otimes I_{\ell}).
%\end{displaymath}
%Therefore, for all $\Span\{p_k\} = \mathcal{P}_{M-1}$, we have $\mathcal{R}(\widetilde{S}) = %\mathcal{R}(S)$.
%Thus, for numerical reason, we shall select $\{p_k\}$ such that $\widetilde{S}$ is well-conditioned.
%Since $F_d(p_k)(A) \approx X_{n_{ev}}p_k(\Lambda_{n_{ev}})h(\Lambda_{n_{ev}})X_{n_{ev}}^T$, it is %sufficient to consider only the performance of $p_k$ within $[a, b]$.
%In \cite{block_SS--RR,SS_theory2010}, the polynomial $p_k$ is selected as $\left( (2t-a-b)/(b-a) %\right)^k$.
%Here, in this paper, we choose $p_k = T_k((2t-a-b)/(b-a))$, the Chebyshev polynomial in the interval $[a, %b]$.

\cref{alg: search subspace} describes the construction of the new basis matrix
\begin{displaymath}
  \widetilde{S} = (F_d(p_0)(A)V, F_d(p_1)(A)V, \ldots, F_d(p_{M-1})(A)V),
\end{displaymath}
where we need to compute the Chebyshev coefficients of $p_k, k = 0, \ldots, M-1$:
\begin{equation}\label{eq: integral form of c_{k,j}}
  c_{k,j} = \frac{2}{\pi} \int_a^b \frac{p_k(t) T_j(t)}{\sqrt{1-t^2}} {\rm d} t.
\end{equation}
In computations, we compute them by the MATLAB built-in function \texttt{integral} in the MATLAB computing environment.

\begin{algorithm}
  \caption{The construction of $\widetilde{S}$}
  \label{alg: search subspace}
  \begin{algorithmic}[1]
    \REQUIRE The matrix $A$, the interval $[a,b]$, the series degree $d$, the moment number $M$, Jackson factors $\{ \rho_{j,d} \}$, Chebyshev coefficients $\{ c_{k,j} \}$, and the $n$-by-$\ell$ matrix $V$ with $M\ell \geq n_{ev}$.
    \STATE Set $V_0 = V$, $V_1 = A V_0$.
    \FOR {$k = 0, 1, \ldots, M-1$}
      \STATE $\widetilde{S}_k = \frac{1}{2} c_{k,0} V_0 + \rho_{1,d} c_{k,1} V_1$.
    \ENDFOR
    \FOR {$j = 2, 3, \ldots, d$}
      \STATE $V_j = 2AV_{j-1} - V_{j-2}$.
      \FOR {$k = 0, 1, \ldots, M-1$}
        \STATE $\widetilde{S}_k \gets \widetilde{S}_k + \rho_{j,d} c_{k,j} V_j$.
      \ENDFOR
    \ENDFOR
    \RETURN $\widetilde{S} = (\widetilde{S}_0, \widetilde{S}_1, \ldots, \widetilde{S}_{M-1})$.
  \end{algorithmic}
  \end{algorithm}

\subsection{The estimate for the number of desired eigenvalues}
By \cref{thm: SS--RR exact}, the condition $M \ell \geq n_{ev}$ is indispensable for computing
the $n_{ev}$ desired eigenpairs.
But the value of $n_{ev}$ is generally unknown in practice. Therefore, we have to estimate $n_{ev}$ in advance.

Note that
\begin{displaymath}
  \Trace(h(A)) = \Trace(X_{n_{ev}}^Th(\Lambda_{n_{ev}})X_{n_{ev}}) = n_{ev} - \frac{1}{2} \left\vert \lambda(A) \cap \{a, b\} \right\vert,
\end{displaymath}
Combine the above relation with $F_d(1)(A) \approx h(A)$, and suppose that the eigenvalues of $A$ equal to $a$ or $b$ are simple. Then
\begin{displaymath}
  n_{ev} \approx \Trace(F_d(1)(A)) + 1.
\end{displaymath}
We estimate the trace by Monte-Carlo simulation \cite{trace-estimate2011,determinants-estimate2022,hutchinson1989}.
Given $s$ random vectors $v_i$ with the components $v_{ij}$ being i.i.d. Rademacher random variables, i.e., $v_{ij} = \pm 1$ with probability $1/2$,
we estimate $n_{ev}$ using
\begin{displaymath}% \label{eq: estimate nev}
  \tilde{n}_{ev} = \frac{1}{s} \sum_{i=1}^s v_i^TF_d(1)(A)v_i + 1.
\end{displaymath}
% We summarize it as \cref{alg: estimate nev}.
% \begin{algorithm}
% \caption{Estimation of $n_{ev}$}
% \label{alg: estimate nev}
% \begin{algorithmic}[1]
%   \REQUIRE The matrix $A$, the interval $[a,b]$, the series degree $d$, and $s$ Rademacher random $n$-vectors $v_1, \cdots, v_s$.
%   \STATE Compute $\tilde{n}_{ev} = \frac{1}{s} \sum_{i=1}^sv_i^TF_d(1)(A)v_i + 1$.
%   \RETURN $\tilde{n}_{ev}$.
% \end{algorithmic}
% \end{algorithm}
More details of the estimation can be found in \cite{CJ-FEAST-cross}.

\subsection{A practical choice of degree \texorpdfstring{$d$}{d} of the CJ series expansion}
\label{choiced}

We now investigate how to select the degree $d$ of the CJ series expansion to ensure
that $\mu_i/\nu_i < \zeta$ for a given $\zeta \in (0,1)$.
The following derivations apply to $i=1,\ldots,n_{ev}$. Therefore, we only need to consider
$i=1$, and drop the subscripts for brevity. For $M \geq 2$, \eqref{consts} means
\begin{equation}\label{i1}
  \begin{cases}
    (1+\zeta)\hat{\gamma} + \hat{\delta} + \zeta\hat{\eta} < \frac{1}{2}\zeta - \varepsilon, \quad
  (1+\zeta) (\hat{\gamma} + \hat{\eta}) < \frac{1}{2} (\zeta - \varepsilon), & \lambda_1 = b, \\
  (1 + \zeta) (\hat{\gamma} + \hat{\delta}) < \zeta - \varepsilon, \quad
  (1 + \zeta) \hat{\gamma} + \zeta \hat{\delta} + \hat{\eta} < \zeta - \frac{1}{2}\varepsilon, & \lambda_1 \neq b.
  \end{cases}
\end{equation}
From \eqref{consts}, suppose that $d$ is sufficiently large so that $\frac{1}{2} > \hat{\gamma} + \hat{\eta} > \hat{\eta}$, which ensures $\nu > 0$.
Then the expressions of $\hat\delta$ and $\hat\eta$ in \eqref{eq: param: raletive errors} mean
$$
\frac{\pi^2(M-1)^2}{(b-a)(d+2)} < \frac{1}{\tau} \leq 1 \quad \text{and thus} \quad \hat{\delta} < \hat{\eta}.
$$
The sufficient conditions for \eqref{i1} are
\begin{displaymath}%\label{eq: original degree estimate for M > 1}
  \begin{cases}
  2 \cdot \max\{\hat{\gamma}, \hat{\eta}\} < \frac{\frac{1}{2}\zeta - \varepsilon}{1 + \zeta}, & \lambda_1 = b, \\
  2 \cdot \max\{\hat{\gamma}, \hat{\delta}\} < \frac{\zeta - \varepsilon}{1 + \zeta}, \quad
  2 \cdot \max\{\hat{\gamma}, \hat{\eta}\} < \frac{\zeta - \frac{1}{2}\varepsilon}{1 + \zeta}, & \lambda_1 \neq b.
  \end{cases}
\end{displaymath}
For $M=1$, according to \cref{rk: M = 1 the subspace iteration}, we have $2\hat{\gamma} < \frac{\zeta}{1+\zeta}$.
% \begin{displaymath}%\label{eq: original degree estimate for M = 1}
%   2 \cdot \hat{\gamma} < \frac{\zeta}{1+\zeta}.
% \end{displaymath}

Substituting \cref{eq: param: raletive errors} into the above relations yields
\begin{equation}\label{eq: exact degree estimate}
  d + 2 > \begin{cases}
    \frac{\pi^2}{\Delta_{\min}^{4/3}}\sqrt[3]{(1+\zeta)/\zeta}, & M = 1, \\
    \max\left\{ \frac{\pi^2}{\Delta_{\min}^{4/3}}\sqrt[3]{\frac{1+\zeta}{\zeta/2-\varepsilon}}, \frac{\pi^2(M-1)^2}{b-a} \cdot \frac{1+\zeta}{\zeta/2-\varepsilon} \right\}, & M \geq 2, \lambda_1 = b, \\
    \max\left\{ \frac{\pi^2}{\Delta_{\min}^{4/3}}\sqrt[3]{\frac{1+\zeta}{\zeta-\varepsilon} \tau}, \frac{\pi^2(M-1)^2}{b-a} \cdot \max\{\sqrt{\frac{1+\zeta}{\zeta-\varepsilon}\tau}, \frac{1+\zeta}{\zeta-\varepsilon/2}\tau\} \right\}, & M \geq 2, \lambda_1 \neq b,
  \end{cases}
\end{equation}
where, by \cref{eq: u_d fourth moment asymptotic}, we rewrite $\hat{\gamma}$ as $\hat{\gamma} = \frac{\pi^6}{\Delta_{\min}^4(d+2)^3} \cdot \frac{\tau}{2}$, and $\tau=1$ when $M=1$ or $\lambda_1=b$.

Without a priori knowledge of the eigenvalue distribution, suppose that the desired
$\lambda_i$ are uniformly distributed in $[a, b]$.
Then
\begin{align*}
  & \Delta_{\min} \approx \min\left\{ \left\vert \arccos(a \pm \frac{b-a}{n_{ev}}) - \alpha \right\vert, \ \left\vert \arccos(b \pm \frac{b-a}{n_{ev}}) - \beta \right\vert \right\} \geq \frac{b-a}{n_{ev}}, \\
  & \varepsilon = 1 / T_{M-1} \left( \frac{n_{ev}-1+\ell}{n_{ev}-1-\ell} \right), \quad
   \tau = T_{M-1} \left( \frac{n_{ev}+1+\ell}{n_{ev}-1-\ell} \right) / T_{M-1} \left( \frac{n_{ev}-1+\ell}{n_{ev}-1-\ell} \right).
\end{align*}
Since $M\ell \geq n_{ev}$, for $M \geq 2$ we have
\begin{displaymath}
  \frac{n_{ev}-1+\ell}{n_{ev}-1-\ell} \geq \frac{n_{ev}+n_{ev}/M}{n_{ev}-n_{ev}/M} = \frac{M+1}{M-1},
\end{displaymath}
which shows
\begin{displaymath}
  \varepsilon \leq 1 / T_{M-1} \Big( \frac{M+1}{M-1} \Big) \leq \frac{1}{3}.
\end{displaymath}
Thus, setting the maximum $\zeta = 1$, we obtain a sufficient condition for \cref{eq: exact degree estimate}:
\begin{equation}\label{eq: degree estimate}
  d + 2 > \max \left\{ \frac{\pi^2n_{ev}^{4/3}}{(b-a)^{4/3}} \max\{ \sqrt[3]{12}, \sqrt[3]{3\tau} \}, \ \frac{\pi^2(M-1)^2}{b-a} \max\{ 12, \frac{12\tau}{5} \} \right\}.
\end{equation}

However, as illustrated in \cref{fig: accuracy of q_d approximating g}, the above bound, especially its first term, is overestimated; that is, the estimated degree $d$ may be too large to be feasible in practice.
We will propose a practical estimation of $d$ as follows: Choose
\begin{equation}\label{eq: degree estimate for practice}
  d = \left\lceil \frac{D\pi^2}{(b-a)^{4/3}} + \frac{\pi^2(M-1)^2}{K^2(b-a)} \right\rceil - 2,
\end{equation}
where $D \in [1, 8]$ and $K \in [1, 10]$.

\subsection{A complete algorithm}
With the algorithmic implementation details described previously, we can run the CJ--SS--RR algorithm
for given $M$ and $\ell$. If it converges, we stop the algorithm and output the converged
approximate eigenpairs. Otherwise, for $M\ell$ fixed, it is necessary to restart the algorithm
to compute better approximations until convergence.
In this paper, we use the restart approach proposed in \cite{SS_theory2016,SS_review}:
At the $k$th restart, we take $V^{(k)} = F_d(1)(A)V^{(k-1)}$, or equivalently, the first $\ell$ columns of
$\widetilde{S}^{(k-1)}$.
We use the stopping criteria in \cite{FEAST_oblique} to test convergence, and describe the
complete algorithm as \cref{alg: CJSSRR}.
The computational cost of one iteration of \cref{alg: CJSSRR} is listed in \cref{tab: computational cost}, where MVs are the number of matrix-vector products and other flops are the flops spent on the other
computations.

\begin{algorithm}
\caption{The restarted CJ--SS--RR algorithm}
\label{alg: CJSSRR}
\begin{algorithmic}[1]
  \REQUIRE The matrix $A$, the interval $[a,b]$, the series degree $d$, the moment number $M$, and an $n$-by-$\ell$ matrix $V^{(0)}$ with $M\ell \geq n_{ev}$.
  \STATE Compute the CJ coefficients by \cref{eq: Jackson damping factor,eq: integral form of c_{k,j}}.
  \FOR {$k = 1, 2, \ldots$}
    \STATE Compute $\widetilde{S}^{(k)}$ by \cref{alg: search subspace}.\label{step: S}
    \STATE Compute the QR decomposition: $\widetilde{S}^{(k)} = U^{(k)}R^{(k)}$.\label{step: QR}
    \STATE Apply the Rayleigh--Ritz projection \cref{eq: Rayleigh--Ritz} to get the Ritz pairs $(\tilde{\lambda}_i^{(k)}, \tilde{x}_i^{(k)}), i = 1, 2, \ldots, M\ell$.\label{step: RR}
    \STATE Compute the residual norms of the Ritz pairs, and test convergence.\label{step: res}
    \STATE Set $V^{(k+1)}$ to be the first $\ell$ columns of $\widetilde{S}^{(k)}$.
  \ENDFOR
  \RETURN The $n_{ev}$ converged approximate eigenpairs $(\tilde{\lambda}_i^{(k)}, \tilde{x}_i^{(k)})$.
\end{algorithmic}
\end{algorithm}

\begin{table}[ht]
  % \footnotesize
  \caption{The computational cost of one iteration of \cref{alg: CJSSRR}}
  \label{tab: computational cost}
  \begin{center}
  \begin{tabular}{lll}
    \toprule
    Steps & MVs & Other flops \\
    \midrule
    \ref{step: S}  & $d \ell$ & $2d n \ell + 2d n M\ell$ \\
    \ref{step: QR} &          & $2n (M\ell)^2$ \\
    \ref{step: RR} & $M\ell$  & $4n (M\ell)^2 + 8\frac{2}{3}(M\ell)^3$ \\
    \ref{step: res} &          & $2n (M\ell)^2$ \\
    Total cost & $(\frac{d}{M}+1)M\ell$ & $2(d+\frac{d}{M}) n M\ell + 8n (M\ell)^2 + 8\frac{2}{3}(M\ell)^3$ \\
    \bottomrule
  \end{tabular}
  \end{center}
\end{table}

Let $nnz(A)$ be the number of nonzero entries in $A$ and
suppose $d>M\ell$; in fact, generally $d \gg M\ell$ in computations. Notice that $n\gg M\ell$.
Then it is seen from \cref{tab: computational cost} that one iteration
costs $O(nnz(A) d\ell) + O(dnM\ell) + O(n(M\ell)^2)$ flops, where the first two terms
consumed by Step \ref{step: S} dominate the overall cost.
As a consequence, we use the cost of Step \ref{step: S} to represent the overall cost of one iteration.
For ease of counting and comparison, we measure the computational cost in terms of MVs.
Since one matrix-vector product of $A$ costs $2nnz(A)$ flops and
\begin{equation}\label{mvsmore}
  \frac{2dn\ell + 2dnM\ell}{2nnz(A)} = \frac{(M+1)n}{nnz(A)} \cdot d \ell,
\end{equation}
the cost of the other flops in Step \ref{step: S} is approximately $ \frac{(M+1)n}{nnz(A)} d \ell $ MVs,
which are not necessarily integers.
We call them the equivalent MVs.
As a result, we measure the overall efficiency of \cref{alg: CJSSRR} by the total MVs, i.e.,
the actual MVs of one iteration plus the equivalent MVs of Step \ref{step: S}.

% Suppose $d>M\ell$; in fact, generally $d \gg M\ell$ in computations.
% Then it is seen from \cref{tab: computational cost} that one iteration costs $O(dnM\ell) + O(n(M\ell)^2)$ other flops, where the first term consumed by Step \ref{step: S} dominates the other flops.
% As a consequence, we use the other flops of Step \ref{step: S} to represent the other flops of one iteration.
% Let $nnz(A)$ be the number of nonzero entries in $A$.
% Then one matrix-vector product of $A$ costs $2nnz(A)$ flops and
% \begin{displaymath}
%   \frac{2dn\ell + 2dnM\ell}{2nnz(A)} = \frac{(M+1)n}{nnz(A)} \cdot d \ell.
% \end{displaymath}
% Thus the cost of the other flops in Step \ref{step: S} is approximately $ \frac{(M+1)n}{nnz(A)} d \ell $ MVs.
% We call them the equivalent MVs, which are not necessarily integers.
% As a result, we measure the overall efficiency of \cref{alg: CJSSRR} by the total MVs, i.e., the actual MVs of one iteration plus the equivalent MVs of Step \ref{step: S}.

\section{Numerical experiments}
\label{sec: numerical experiments}

We report numerical experiments to confirm our theoretical results and illustrate the performance of
\cref{alg: CJSSRR}.
The test matrices are from the SuiteSparse Matrix Collection \cite{matrix_collection}, and we list some of
their basic properties and the intervals of interest in \cref{tab: test matrices}.
%Bounding the spectrum of $A$ and estimating the number $n_{ev}$ are not the focus of this paper, and
%they have been considered in detail in \cite{CJ-FEAST-cross,spectral-estimate2011zhou}.
All the numerical experiments were performed on an Intel Core i7-9700, CPU 3.0GHz, 8GB RAM using MATLAB R2024b with the machine precision $\epsilon_{\rm mach} = 2.22 \times 10^{-16}$ under the Microsoft Windows 10 64-bit system.
To make a fair comparison, for each test problem and given $\ell$, we used the same starting $n \times \ell$ matrix $V^{(0)}$ for the algorithms tested,
which was generated randomly in a normal distribution. 
For reference, we use the
built-in function \texttt{eig} in MATLAB to compute $\lambda_{\min}$, $\lambda_{\max}$ and the exact $n_{ev}$ 
and obtain $\norm{A}=\max\{\abs{\lambda_{\min}}, \abs{\lambda_{\max}}\}$.

\begin{table}[ht]
  \caption{Properties of the test matrices, where $\lambda_{\min}$, $\lambda_{\max}$ and $n_{ev}$ are computed by \texttt{eig}.}
  \label{tab: test matrices}
  \centering
  \begin{tabular}{lcccccc}
    \toprule
    Matrix A          & Size  & $nnz(A)$  & $\lambda_{\min}$ & $\lambda_{\max}$ & $[a, b]$   & $n_{ev}$ \\
    \midrule
    SiH4              & 5041  & 171903  & -0.9956    & 36.79 & $[3.7, 4.7]$ & 115 \\
    SiNa              & 5743  & 198787  & -0.7042    & 25.62 & $[2, 2.8]$   & 134 \\
    delaunay\_n13     & 8192  & 49094   & -3.564     & 6.457 & $[2.4, 2.8]$ & 214 \\
    benzene           & 8219  & 242669  & -0.7297    & 58.40 & $[4, 5.1]$   & 115 \\
    stokes64          & 12546 & 140034  & -1.696e-3  & 3.998 & $[1.9, 2.1]$ & 232 \\
    Pres\_Poisson     & 14822 & 715804  &  1.278e-5  & 26.03 & $[10, 12]$   & 25  \\
    Si10H16           & 17077 & 875923  & -1.149     & 36.93 & $[3.7, 4.7]$ & 371 \\
    brainpc2          & 27607 & 179395  & -2000      & 4465  & $[-10, 10]$  & 88  \\
    rgg\_n\_2\_15\_s0 & 32768 & 320480  & -5.119     & 17.36 & [10, 11]   & 221 \\
    SiO               & 33401 & 1317655 & -1.675     & 84.32 & $[2, 4]$     & 332 \\
    \bottomrule
  \end{tabular}
\end{table}

Based on \cref{sec: CJ--SS--RR algorithm}, in the experiments, we apply the linear transformation $l(t)$ in \cref{eq: linear
transformation} to $A$ such that the spectrum interval of $l(A)$ is $[-1, 1]$ and
$[a, b]$ is mapped to $[\tilde{a}, \tilde{b}] \subset [-1, 1]$,
where $\tilde{a} = l(a)$ and $\tilde{b} = l(b)$.
An approximate eigenpair $(\tilde{\lambda}, \tilde{x})$ is claimed to have converged if its relative residual norm satisfies
\begin{equation}\label{stopcrit}
  \frac{\norm{A\tilde{x} - \tilde{\lambda}\tilde{x}}}{\norm{A}\norm{\tilde{x}}} < tol.
\end{equation}

\subsection{Condition numbers of different bases and effective dimensions of the search subspaces}
\label{subsec: condition number of S}

We first investigate the condition numbers of different kinds of matrices $S$ and effective
dimensions of $\mathcal{S}$
constructed by the CJ series approximations to the moments corresponding to different kinds of
polynomials $p_k(t), k=0, 1, \ldots, M-1$,
where the numerical ranks, i.e., the effective dimensions of $\mathcal{R}(S)$'s,
are determined by the truncated SVDs with the tolerance $10\|S\|\epsilon_{\rm mach}$.
%The authors of \cite{block_SS--RR,SS_theory2010} have realized the ill conditioning of $S$ generated by the monomials $1, t, %\ldots, t^{M-1}$ and suggested to
%replace them by the shifted-and-scaled $(\tilde{l}(t))^k$ with $\tilde{l}(t) = (2t-\tilde{a}-\tilde{b})/(\tilde{b}-\tilde{a})$; but they did not pay any attention to the effective dimensions of $\mathcal{S}$ and their vital effects on the performance of the method.
\cref{tab: numerical rank of S,tab: condition number of S} display the condition numbers and effective subspace dimensions
obtained by using the monomial moments, the shifted-and-scaled $(\tilde{l}(t))^k$ with $\tilde{l}(t) = (2t-\tilde{a}-\tilde{b})/(\tilde{b}-\tilde{a})$ and the shifted-and-scaled Chebyshev moments for the first five matrices in \cref{tab: test matrices}
with $M = 4, 8, 16$,
where, for each test matrix, we set the true dimension $M\ell=16 \lceil 1.5 n_{ev} / 16 \rceil$  of $\mathcal{R}(S)$, and determined $\ell$ by the equations $M\ell=16 \lceil 1.5 n_{ev} / 16 \rceil$. We took the series degree $d$ in \cref{eq: degree estimate for practice} as $d = \lceil \pi^2 (b-a)^{-4/3} + \pi^2 (b-a)^{-1} \rceil$.

As is seen from the tables,
the condition numbers of $S$'s using the monomial moments are already bigger than
$O(1/\sqrt{\epsilon_{\rm mach}})$ even for $M=4$,
and for $M=8$ these $S$ are already numerically rank deficient.
For $M=16$,
the effective search subspaces obtained by the truncated SVD's
of $S$ loses too much information on $\mathcal{R}(S)$ and effective dimensions of $\mathcal{R}(S)$'s
are already smaller than $n_{ev}$'s, causing that \cref{alg: CJSSRR} cannot compute
the $n_{ev}$ desired eigenpairs of $A$.
% In all the cases, increasing $M$ from 8
% to 16 is of no use since the dimensions of effective
% search subspaces remain the same as those for $M=8$.

In contrast, the situation changes very substantially when using the shifted-and-scaled moments,
and the condition numbers are reduced greatly. However, it is still
far from good enough to solve the concerning
eigenvalue problems because the condition numbers are around $10^{10}$ for $M=16$ and the
errors of the
computed subspaces by the rank-revealing QR decompositions and the true subspaces
$\mathcal{R}(S)$ are the condition numbers times $\epsilon_{\rm mach}$ and thus around $10^{-5}$.

Strikingly, the shifted-and-scaled Chebyshev moments is very superior to the shifted-and-scaled moments,
and the condition numbers are further significantly reduced and are around
$10^5$, approximately the square roots of the
latter ones for $M=16$. As a result, the computed subspaces are much more accurate with
the errors $10^{-10}$, so that we can compute the desired eigenpairs much more accurately.
For better stability and accuracy, we always use
the Chebyshev polynomials $T_k(\tilde{l}(t))$ to generate $S$ in the subsequent experiments.

\begin{table}[ht]
  \caption{The condition numbers of the moment matrices generated by different polynomials $p_k(t)$.}
  \label{tab: condition number of S}
  \centering
  \begin{tabular}{*{7}{c}}
    \toprule
    \multirow{2}{*}{$M$} & \multirow{2}{*}{$p_k(t)$} & \multicolumn{5}{c}{The matrices} \\ \cline{3-7}
    &                           & SiH4            & SiNa            & delaunay\_n13   & benzene         & stokes64        \\ \midrule
    \multirow{3}{*}{4}  & $t^k$ & 2.5e+10         & 2.1e+10         & 3.6e+8          & 1.9e+11         & 5.6e+7          \\
    &      $( \tilde{l}(t) )^k$ & 1.8e+5          & 2.0e+5          & 2.6e+4          & 3.4e+5          & 9.5e+3          \\
    &     $T_k( \tilde{l}(t) )$ & \textbf{5.7e+4} & \textbf{5.3e+4} & \textbf{7.8e+3} & \textbf{1.2e+5} & \textbf{3.4e+3} \\ \midrule
    \multirow{3}{*}{8}  & $t^k$ & 7.0e+15         & 6.0e+15         & 5.7e+15         & 4.4e+15         & 2.1e+14         \\
    &      $( \tilde{l}(t) )^k$ & 7.8e+6          & 6.6e+6          & 9.5e+5          & 2.0e+7          & 3.6e+5          \\
    &     $T_k( \tilde{l}(t) )$ & \textbf{1.2e+5} & \textbf{1.1e+5} & \textbf{1.2e+4} & \textbf{3.3e+5} & \textbf{5.3e+3} \\ \midrule
    \multirow{3}{*}{16} & $t^k$ & 6.6e+16         & 5.1e+16         & 8.1e+22         & 2.0e+16         & 7.0e+28         \\
    &      $( \tilde{l}(t) )^k$ & 7.7e+9          & 8.6e+9          & 6.7e+9          & 2.6e+10         & 1.6e+10         \\
    &     $T_k( \tilde{l}(t) )$ & \textbf{1.8e+5} & \textbf{2.0e+5} & \textbf{1.8e+5} & \textbf{5.4e+5} & \textbf{3.8e+5} \\
    \bottomrule
  \end{tabular}
\end{table}

\begin{table}[ht]
  \caption{The numerical ranks of the moment matrices generated by different polynomials $p_k(t)$,
  where the boldface numbers denote numerical rank deficiency.}
  \label{tab: numerical rank of S}
  \centering
  \begin{tabular}{*{7}{c}}
    \toprule
    \multirow{2}{*}{$M$} & \multirow{2}{*}{$p_k(t)$} & \multicolumn{5}{c}{The matrices} \\ \cline{3-7}
    &                            & SiH4         & SiNa         & delaunay\_n13 & benzene      & stokes64     \\ \midrule
    \multirow{3}{*}{4}  & $t^k$  & 176          & 208          & 336           & 176          & 352          \\
    &      $( \tilde{l}(t) )^k$  & 176          & 208          & 336           & 176          & 352          \\
    &     $T_k( \tilde{l}(t) )$  & 176          & 208          & 336           & 176          & 352          \\ \midrule

    \multirow{3}{*}{8}  & $t^k$  & \textbf{134} & \textbf{163} & \textbf{296}  & \textbf{132} & \textbf{344} \\
    &      $( \tilde{l}(t) )^k$  & 176          & 208          & 336           & 176          & 352          \\
    &     $T_k( \tilde{l}(t) )$  & 176          & 208          & 336           & 176          & 352          \\ \midrule

    \multirow{3}{*}{16} & $t^k$  & \textbf{88}  & \textbf{104} & \textbf{168}  & \textbf{79}  & \textbf{176} \\
    &      $( \tilde{l}(t) )^k$  & 176          & 208          & 336           & 176          & 352          \\
    &     $T_k( \tilde{l}(t) )$  & 176          & 208          & 336           & 176          & 352          \\ \midrule
    \multicolumn{2}{l}{$M \ell$} & 176          & 208          & 336           & 176          & 352          \\
    \multicolumn{2}{l}{$n_{ev}$} & 115          & 134          & 214           & 115          & 232          \\
    \bottomrule
  \end{tabular}
\end{table}

\subsection{The performance of the CJ--SS--RR algorithm}

% Relation \eqref{eq: degree estimate} show that, the bigger $M$ is, the bigger the series
% degree $d$ is required to achieve a comparable error in \eqref{eq: bound for approx Krylov bound} in size.
We test the matrix stokes64
% matrices delaunay\_n13 and stokes64
for $M = 1, 2, 4, 8, 16$ by fixing the subspace dimension $M \ell = 16 \lceil 1.5 n_{ev} / 16 \rceil$ and using the stopping tolerance $tol = 10^{-10}$ in \eqref{stopcrit}. 
In the experiments, two options are available for the series degree $d$.
The first is to use the constant degree $d = \lceil \pi^2 (b-a)^{-4/3} + \pi^2 (b-a)^{-1} \rceil$ for all
the moment matrices, and the second is to use the changing series
degree $d$ determined by \cref{eq: degree estimate for practice} with $D = 1$ and $K = 10$ so that the
moment matrices are generated using the CJ series expansions of the changing degree $d$.
The number of restarts denoted by Iter, the maximum relative residuals and the total MVs are listed in \cref{tab: CJ--SS--RR on stokes64}.

% \begin{table}[ht]
%   \caption{The performance of \cref{alg: CJSSRR} applied on the test matrix delaunay\_n13.}
%   \label{tab: CJ--SS--RR on delaunay_n13}
%   \centering
%   \begin{tabular}{*{10}{c}}
%     \toprule
%     \multirow{2}{*}{$M$} & \multicolumn{4}{c}{Using the constant $d$} &
%     & \multicolumn{4}{c}{Using the changing $d$} \\ \cline{2-5} \cline{7-10}
%        & $d$  & Iter. & Residual & Total MVs & & $d$ & Iter. & Residual & Total MVs \\
%     \midrule
%     1  & 409  & 5     & 2.75e-11 & 885318.8  & & 286 & 6     & 7.03e-12 & 743494.8  \\
%     2  & 409  & 5     & 5.50e-12 & 492629.1  & & 287 & 6     & 3.75e-11 & 415421.5  \\
%     4  & 409  & 5     & 1.16e-11 & 296284.2  & & 297 & 6     & 1.63e-11 & 258732.3  \\
%     8  & 409  & 5     & 1.34e-11 & 198111.8  & & 346 & 5     & 2.37e-12 & 167854.6  \\
%     16 & 409  & 8     & 1.85e-11 & 238441.0  & & 564 & 4     & 3.09e-11 & 163892.5  \\
%     \bottomrule
%   \end{tabular}
% \end{table}

\begin{table}[ht]
  \caption{The performance of \cref{alg: CJSSRR} on the test matrix stokes64.}
  \label{tab: CJ--SS--RR on stokes64}
  \centering
  \begin{tabular}{*{10}{c}}
    \toprule
    \multirow{2}{*}{$M$} & \multicolumn{4}{c}{Using the constant $d$} &
    & \multicolumn{4}{c}{Using the changing $d$} \\ \cline{2-5} \cline{7-10}
       & $d$ & Iter & Residual & Total MVs & & $d$ & Iter & Residual & Total MVs \\
    \midrule
    1  & 310 & 5     & 1.34e-12 & 642345.0  & & 211 & 8     & 3.00e-11 & 700433.8  \\
    2  & 310 & 5     & 1.49e-12 & 345798.8  & & 212 & 8     & 1.71e-11 & 379261.0  \\
    4  & 310 & 5     & 9.70e-13 & 197525.7  & & 220 & 8     & 1.88e-11 & 225104.7  \\
    8  & 310 & 5     & 2.84e-12 & 123389.1  & & 259 & 7     & 3.30e-12 & 144730.8  \\
    16 & 310 & 12    & 6.92e-11 & 207169.9  & & 433 & 5     & 3.73e-12 & 119872.4  \\
    \bottomrule
  \end{tabular}
\end{table}

Relation \cref{eq: degree estimate for practice} indicates that the degree $d$ has to increase quadratically with $M$
so as to ensure a similar convergence.
Consequently, the fixed constant degree $d$ may slow down convergence when $M>1$ considerably.
The results in \cref{tab: CJ--SS--RR on stokes64}
confirm the considerable slowdown of convergence of the constant $d$ for $M=16$.
The results also demonstrate that for a given $M$, a higher degree $d$ leads to a faster convergence, and vice versa, which is consistent with \cref{rk: convergent analysis for M and d}.

\Cref{tab: CJ--SS--RR on stokes64}
illustrate that the total MVs used by \cref{alg: CJSSRR} are notably fewer for $M = 2$ than for $M = 1$.
On the contrary, the total MVs for $M = 8$ and $M = 16$ are comparable.
This can be explained by the computational cost list in \cref{tab: computational cost}, where,
for relatively large $M$, the reduction in total MVs due to the increase in $M$ is offset by increasing the degree $d$.
Since $d$ increases quadratically with $M$, the total MVs of an excessively large $M$ may instead be bigger.
We thus suggest that a modest value of $M$, e.g.,  $4\sim 8$, be used.

\subsection{A comparison of the CJ--SS--RR algorithm and the SS--RR algorithm}

In this subsection we compare \cref{alg: CJSSRR} with the block SS--RR method \cite{SS_package,block_SS--RR}.
Here, for a given interval $[a, b]$ of interest, the contour is a circle with the center $(a+b)/2$ and radius $(b-a)/2$.
The numerical quadrature rule used is the trapezoidal rule with sixteen nodes,
which are default parameters in \cite{SS_package}.
Notice that $A$ is symmetric and the quadrature nodes are symmetric with respect to the real axis, the block SS--RR method manages to solve only $M\ell \cdot\frac{q}{2}$ linear systems at each iteration, where $q$ is the number of quadrature nodes.
The shifted linear systems are solved by the MATLAB built-in function
\texttt{minres} with the stopping tolerance $10^{-12}$.

We tested the two algorithms for $M = 4$ and $8$ by taking
$M\ell = 8 \lceil 1.5 n_{ev} / 8 \rceil$ to ensure that $M\ell \geq 1.5 n_{ev}$.
The convergence tolerance in \eqref{stopcrit} is $tol = 10^{-12}$.
In \cref{alg: CJSSRR}, we chose the degree $d$ using \cref{eq: degree estimate for practice}
with $D = 2$ and $K = 7$.
% Since the block SS--RR method is the most costly in terms of solving linear systems, only the actual MVs generated by $A$ are recorded as the total MVs.
Since the generally dominating cost of the contour-based block SS--RR method is iterative solutions of
shifted linear systems, we only record the total MVs with $A$ used by \texttt{minres}, which means that 
we ignore other costs of the contour-based block SS--RR method when comparing the overall efficiency.

\begin{table}[ht]
  \caption{The results of block SS--RR and \cref{alg: CJSSRR} with $tol = 10^{-12}$.
  The stopping tolerance of \texttt{minres} in block SS--RR is $10^{-12}$.
  Speedup ratios are the ratios of MVs used by \cref{alg: CJSSRR} and those used by
   \texttt{minres}.}
  \label{tab: comparison with SS--RR}
  \centering
  \begin{tabular}{l*{6}{c}rr}
    \toprule
    \multirow{3}{*}{Matrix} & \multicolumn{5}{c}{The total MVs} & & \multicolumn{2}{c}{\multirow{2}{*}{Speedup ratios}} \\\cline{2-6}
    & \multicolumn{2}{c}{SS--RR} & & \multicolumn{2}{c}{CJ--SS--RR} &  & \multicolumn{2}{c}{} \\
    \cline{2-3} \cline{5-6} \cline{8-9}
                      & $M = 4$ & $M = 8$ & & $M = 4$ & $M = 8$ & & \multicolumn{1}{c}{$M=4$} & \multicolumn{1}{c}{$M=8$} \\
    \midrule
    SiH4              & 1.1e+6  & 7.6e+5  & & 2.1e+5  & 9.9e+4  & &  5.1    &  7.6    \\
    SiNa              & 3.2e+6  & 1.7e+6  & & 2.0e+5  & 1.3e+5  & & 15.9    & 12.7    \\
    delaunay\_n13     & 1.3e+7  & 6.7e+6  & & 5.0e+5  & 3.3e+5  & & 26.1    & 20.3    \\
    benzene           & 2.5e+6  & 1.8e+6  & & 3.4e+5  & 2.1e+5  & &  7.3    &  8.3    \\
    stokes64          & 8.0e+6  & 4.1e+6  & & 2.8e+5  & 2.1e+5  & & 28.5    & 19.7    \\
    Pres\_Poisson     & 7.7e+4  & 3.9e+4  & & 1.4e+4  & 7.3e+3  & &  5.5    &  5.3    \\
    Si10H16           & 2.2e+7  & 1.1e+7  & & 6.4e+5  & 3.9e+5  & & 33.8    & 27.5    \\
    brainpc2          & 1.6e+7  & 1.1e+7  & & 3.1e+6  & 2.2e+6  & &  5.0    &  5.0    \\
    rgg\_n\_2\_15\_s0 & 9.0e+6  & 5.8e+6  & & 3.2e+5  & 2.4e+5  & & 28.1    & 24.6    \\
    SiO               & 9.8e+6  & 4.9e+6  & & 6.9e+5  & 4.3e+5  & & 14.2    & 11.4    \\
    \bottomrule
  \end{tabular}
\end{table}

We record the total MVs of both algorithms and the speedup ratios of \cref{alg: CJSSRR} over
the block SS--RR method in \cref{tab: comparison with SS--RR}.
As we see from \cref{tab: comparison with SS--RR}, all the speedup ratios are no less than 5, indicating that the CJ--SS--RR algorithm is significantly more efficient than the block SS--RR method.
For Pres\_Poisson, only 47 of the total 14822 eigenvalues are on the right of the interval of interest,
which means that the desired eigenvalues are near to the largest ones and the shifted linear systems
are slightly indefinite, so that MINRES has relatively quick convergence.
In this case, the speedup ratios are relatively not so substantial, though already quite considerable.
However, for stokes64, Si10H16 and rgg\_n\_2\_15\_s0, where the intervals of interest correspond to the interior eigenvalues, the speedup ratios of the total MVs are significantly big.
These indicate that when the interval of interest corresponds to interior eigenvalues, the high
indefiniteness of the coefficient matrices of the shifted linear systems in the block SS--RR method leads to much more slow convergence of MINRES.
% A special case is that for test problem rgg\_n\_2\_15\_s0, the block SS--RR method fails to converge until the stopping tolerance of \texttt{minres} is improved to $10^{-13}$.
% This is perhaps due to the fact that numerical quadrature \cref{eq: numerical quadrature} requires a certain precision in the solution of linear systems, which is more difficult than ensure the relative residual norms of the linear systems less than a given tolerance.
% In contrast, the CJ--SS--RR algorithm requires only a suitable series degree $d$ and there is a theoretical guideline for the choice.

\section{Conclusion}
\label{sec: conclusion}

We have studied the problem of approximating the piecewise continuous function $p(t)h(t)$ by the CJ
series expansion and proved its pointwise convergence to the function $p(t)h(t)$, showing that
the convergence rates critically depend on the point positions and the degree of $p(t)$.
Making use of these results, we have extended the theoretical results of the block Krylov subspace
to the subspace constructed by the CJ series expansion.
Based on them, we have developed a practical restart CJ--SS--RR
algorithm, where we have replaced
ill-conditioned basis vectors of the search subspace generated by the monomial moments by
the much better conditioned Chebyshev basis vectors, which
enable us to make full of the search subspace and improve the robustness of
the SS--RR method and computational accuracy very greatly.
In the meantime, we have proposed a practical selection strategy of the series degree $d$.
%We have analyzed the selection of polynomials used to construct $S$ and counted the computational cost of the %algorithm.

We have numerically tested our algorithm on several problems to confirm the theoretical results and analysis, and shown that the CJ--SS--RR algorithm is effective and efficient.
The experiments have demonstrated that the CJ--SS--RR algorithm outperforms the block SS--RR method
and often improve the overall efficiency dozens of times in terms of the total MVs, especially when the interval of interest corresponds to interior eigenvalues. In summary, the developed CJ--SS--RR algorithm is a much more
efficient and robust alternative of the block SS--RR method when computing the eigenpairs of large sparse real symmetric matrices.

\section*{Declarations}
The authors declare that they have no financial interests,
and they read and approved the final manuscript.
The algorithmic Matlab code is available upon reasonable
request from the corresponding author.

%\bibliographystyle{siamplain}
%\bibliography{references}

\end{document}